\documentclass[12pt]{article}
\usepackage{indentfirst}
\usepackage{amsmath,amssymb,amsthm, amsfonts}
\usepackage{hyperref}
\usepackage{graphicx}
\usepackage{array, tabulary}
\usepackage{url}
\usepackage[mathlines]{lineno}
\usepackage{dsfont} % needed for double-struck 1
\usepackage{color}
\usepackage{subcaption}
\usepackage{enumitem}
\usepackage[dvipsnames]{xcolor}
\definecolor{red}{rgb}{1,0,0}

\definecolor{blue}{rgb}{0,0,1}

\definecolor{green}{rgb}{0,.6,0}

\usepackage{float}

\usepackage{tikz}

\newtheorem{thm}{Theorem}[section]
\newtheorem{cor}[thm]{Corollary}
\newtheorem{lem}[thm]{Lemma}
\newtheorem{prop}[thm]{Proposition}

\theoremstyle{definition}
\newtheorem{rem}[thm]{Remark}

\theoremstyle{definition}
\newtheorem{defn}[thm]{Definition}

\theoremstyle{definition}
\newtheorem{ex}[thm]{Example}

\newcommand{\ZZ}{\mathbb{Z}}

\newcommand{\D}{\Gamma}

\newcommand{\F}{\mathcal{F}}
\newcommand{\E}{\mathcal{E}}

\newcommand{\oti}{\operatorname{OTI}}

\newcommand{\Z}{\operatorname{Z}}

\newcommand{\rev}{\operatorname{Rev}}
\newcommand{\term}{\operatorname{Term}}

\newcommand{\bit}{\begin{itemize}}
\newcommand{\eit}{\end{itemize}}
\newcommand{\beq}{\begin{equation}}
\newcommand{\eeq}{\end{equation}}
\newcommand{\bea}{\begin{eqnarray}} 
\newcommand{\eea}{\end{eqnarray}}
\newcommand{\bpf}{\begin{proof}}
\newcommand{\epf}{\end{proof}\ms}
\newcommand{\bmt}{\begin{bmatrix}}
\newcommand{\emt}{\end{bmatrix}}
\newcommand{\ms}{\medskip}

\newcommand{\lc}{\left\lceil}
\newcommand{\rc}{\right\rceil}
\newcommand{\lf}{\left\lfloor}
\newcommand{\rf}{\right\rfloor}

\newcommand{\noi}{\noindent}
\newcommand{\ceil}[1]{\lc #1 \rc}

\newcommand{\beqs}{\begin{equation*}} % * means no number
\newcommand{\eeqs}{\end{equation*}}
\newcommand{\beas}{\begin{eqnarray*}}
\newcommand{\eeas}{\end{eqnarray*}}

\newcommand{\up}[1]{^{(#1)}}
\newcommand{\upc}[1]{^{[#1]}}
\newcommand{\floor}[1]{\lf #1 \rf}
\newcommand{\calf}{\mathcal{F}}

\newcommand{\pt}{\operatorname{pt}}
\newcommand{\ptp}{\operatorname{pt}_{+}}

\newcommand{\throt}{\operatorname{th}}

\newcommand{\thp}{\operatorname{th}_{+}}

\title{Throttling for standard zero forcing on directed graphs}
\author{
Emily Cairncross \thanks{Department of Mathematics, Oberlin College, Oberlin, OH, USA (ecairncr@oberlin.edu)}
\and
Joshua Carlson \thanks{Department of Mathematics and Statistics, Williams College, Williamstown, MA, USA (\{jc31, pjh1, bnk1\}@williams.edu)}
\and 
Peter Hollander \footnotemark[2]
\and 
Benjamin Kitchen \footnotemark[2]
\and 
Emily Lopez \thanks{Department~of Mathematics, University of California at Santa Barbara, Goleta, CA, USA (emily\_e\_lopez@ucsb.edu)}
\and
Ashley Zhuang \thanks{Harvard University, Cambridge, MA, USA (azhuang@college.harvard.edu)}
}

\date{\today}
\begin{document}
\maketitle

\begin{abstract} 
Zero forcing is a process on graphs in which a color change rule is used to force vertices to become blue. The amount of time taken for all vertices in the graph to become blue is the propagation time. Throttling minimizes the sum of the number of initial blue vertices and the propagation time. In this paper, we study throttling in the context of directed graphs (digraphs). We characterize all simple digraphs with throttling number at most $t$ and examine the change in the throttling number after flipping arcs and deleting vertices. We also introduce the \emph{orientation throttling interval (OTI)} of an undirected graph, which is the range of throttling numbers achieved by the orientations of the graph. While the OTI is shown to vary among different graph families, some general bounds are obtained. Additionally, the maximum value of the OTI of a path is conjectured to be achieved by the orientation of a path whose arcs alternate in direction. The throttling number of this orientation is exactly determined in terms of the number of vertices.
\end{abstract}

\noi {\bf Keywords} Information spread, Zero forcing, Propagation time, Throttling

\noi{\bf AMS subject classification} 05C15, 05C20, 05C50, 05C57

%%%%%%%%%%%%%%%%%%%%%%%%%%%%%%%%
\begin{section}{Introduction}\label{intro}
A simple way to model information in a graph is to color each vertex blue (or white) if the information is known (or unknown) at that vertex respectively. \emph{Zero forcing}, introduced in \cite{AIM}, is a process that uses a color change rule to spread information by iteratively changing the color of vertices from white to blue. The \emph{(standard) color change rule} states that if $u$ is a blue vertex and there is a unique white neighbor $w$ of $u$, then $u$ can force $w$ to become blue. Such a force is denoted $u \rightarrow w$. Given an initial coloring of the vertices, the goal of the zero forcing process is to color the entire vertex set of a graph blue by repeatedly performing valid forces. It is natural to attempt to optimize this process by making it as efficient as possible. In this context, there are multiple ways to interpret efficiency, leading to many rich areas of study. 

One way to make zero forcing efficient is to start the process with as few vertices colored blue as possible. All graphs and digraphs (directed graphs) in this paper are simple and the conventional graph theoretic notation and terminology in \cite{Diestel} is used. If $B \subseteq V(G)$ is the initial set of blue vertices in a graph $G$ and it is possible to eventually force each vertex in $V(G)$ blue, then $B$ is called a \emph{zero forcing set} of $G$. The size of a minimum zero forcing set of $G$ is the \emph{zero forcing number}, $\Z(G)$. 

Zero forcing can also be made efficient by reducing the time taken for all vertices to become blue. The following definitions from \cite{Carlson19} make this concept rigorous. With $B \subseteq V(G)$ as the initial set of blue vertices, a set of forces $\calf$ that can be performed in some order until no more valid forces are possible is called a \emph{set of forces of $B$ in $G$}. A set of forces $\calf$ of a subset $B \subseteq V(G)$ can be used to partition $V(G)$ according to time steps starting with $\calf\up{0} := B$. For each $t\geq 0$, $\calf\up{t+1}$ is defined by considering the coloring of $V(G)$ where $\bigcup_{i=0}^t \calf\up{i}$ is blue and $V(G) \setminus \bigcup_{i=0}^t \calf\up{i}$ white. Specifically, given this coloring, $\calf\up{t+1}$ is the set of white vertices $w$ for which there exists a blue vertex $u$ with $(u \rightarrow w) \in \calf$. For simplicity, $\calf\upc{t} := \bigcup_{i=0}^{t} \calf\up{i}$ for each integer $t \geq 0$. Intuitively, $\calf\up{t}$ is the set of  vertices in $V(G)$ that are forced during time step $t$ using $\calf$ and $\calf\upc{t}$ is the set of vertices in $V(G)$ that are blue at time $t$ using $\calf$. The \emph{propagation time of a set of forces $\calf$ in $G$}, denoted $\pt(G; \calf)$, is the smallest integer $t$ such that $\calf\upc{t} = V(G)$. By convention, $\pt(G; \calf) = \infty$ if $\calf$ does not force all vertices in $V(G)$ to become blue. For a subset $B \subseteq V(G)$, \emph{the propagation time of $B$ in $G$}, denoted $\pt(G; B)$, is the minimum value of $\pt(G; \calf)$ over all sets of forces $\calf$ of $B$ in $G$.

In \cite{proptime}, Hogben et al.~optimize zero forcing by studying the minimum propagation time over all minimum zero forcing sets of a graph $G$. This is called the \emph{propagation time} of $G$ and is denoted $\pt(G)$. Then, in \cite{BY}, Butler and Young study the optimal balance between the size of a zero forcing set and its propagation time by introducing the concept of throttling. For a graph $G$ and zero forcing set $B \subseteq V(G)$, $\throt(G;B) := |B| + \pt(G; B)$ and the \emph{throttling number} of $G$ is defined as $\throt(G) = \min\{ \throt(G;B) ~|~ B \subseteq V(G)\}$. Although the zero forcing number is studied in \cite{AIM} as a tool for bounding the nullity of certain matrices associated with a given graph, throttling for zero forcing and its variants is largely a  combinatorial problem. For standard zero forcing, positive semidefinite zero forcing, power domination, and Z-floor forcing, graphs with throttling numbers at most $t$ for an arbitrary integer $t > 0$ have been characterized as particular minors of some larger host graph (see \cite{powerdomthrot, Carlson19, semidefinite, JK19}). Additionally, standard throttling and positive semidefinite throttling have been described as forbidden subgraph problems in \cite{JK19}.

In recent years, zero forcing concepts have been extended to digraphs. A digraph is denoted $\D = (V(\D), E(\D))$, and a simple digraph is a digraph with no parallel arcs or loop arcs. The term \emph{double arcs} is used to refer to a pair of arcs of the form $(u, v)$ and $(v, u)$. An \emph{oriented graph} is a simple digraph with no double arcs, and an \emph{orientation} $\vec{G}$ of a simple, undirected graph $G$ is an oriented graph whose underlying undirected graph is $G$. If $\D$ is a digraph and $(u,v) \in E(\D)$, $u$ is an \emph{in-neighbor} of $v$ and $v$ is an \emph{out-neighbor} of $u$. The set of all in-neighbors and the set of all out-neighbors of a vertex $v$ in a simple digraph $\D$ is denoted as $N_{\D}^-(v)$ and $N_{\D}^+(v)$ respectively where the subscript can be dropped if $\D$ is clear from context. The \emph{in-degree} and \emph{out-degree} of a vertex $v \in V(\D)$ are defined as $|N^-(v)|$ and $|N^+(v)|$ respectively. Furthermore, a \emph{source} is a vertex with in-degree zero, a \emph{sink} is a vertex with out-degree zero. The \emph{(standard) color change rule for simple digraphs} is that if $u$ is a blue vertex and there is a unique white out-neighbor $w$ of $u$, then $u$ can force $w$ to become blue. 

For a simple digraph $\D$, the zero forcing number $\Z(\D)$, a set of forces $\calf$ of a subset $B \subseteq V(\D)$, $\pt(\D; \F)$, $\pt(\D; B)$, $\pt(\D)$, $\throt(\D; B)$, and $\throt(\D)$ are all defined analogously to their undirected counterparts. The parameter $\Z(\D)$ is studied in \cite{ISUREU}, and an upper bound is given for the difference between the zero forcing numbers of two orientations of a given simple graph. Furthermore, the parameters $\pt(\D; \calf)$, $\pt(\D; B)$, and $\pt(\D)$ are studied in \cite{Berliner et al17}. The next natural step is to explore the throttling number of simple digraphs. While $\throt(\D)$ is determined in \cite{Berliner et al17} for a specific type of \emph{Hessenberg path} (see Section \ref{subsec:monochar}), there is much more to be studied.

In this paper, we take a closer look at throttling for simple digraphs. We obtain a variety of results in Section \ref{simpleDigraphsGeneral} about the throttling numbers of simple digraphs in general. Specifically, we show that the throttling number of a simple digraph does not change when all arcs are reversed and we give a structural characterization of all simple digraphs with throttling number at most $t$ for an arbitrary integer $t>0$. In Section \ref{sec:throttlingorientations}, we examine the possible throttling numbers of all orientations of a given simple, undirected graph $G$. To this end, we define the \emph{orientation throttling interval (OTI)} of a simple graph (see Definition \ref{def:OTI}). An upper bound is given for the difference between the throttling numbers of two orientations of an arbitrary simple graph $G$. While the OTI is shown to vary wildly for different graphs, some general bounds and properties are determined. In Section \ref{altpaths}, we focus on the OTI of path graphs. The lowest possible throttling number of a path on $n$ vertices is shown to be $\lceil 2\sqrt{n} - 1 \rceil$. The throttling numbers of specific orientations of paths are determined exactly, and it is conjectured that these orientations achieve the maximum throttling number of paths on $n$ vertices. Finally, in Section \ref{sec:conclusion}, some concluding remarks are made and directions for future work are given.

\end{section}

%%%%%%%%%%%%%%%%%%%%%%%%%%%%%
\section{Throttling for simple digraphs}\label{simpleDigraphsGeneral}

While throttling for undirected graphs has been studied extensively, in this section, we explore the process on directed graphs in general. First, we consider the case when throttling on digraphs is equivalent to that of an undirected graph. The following remark describes this situation.

\begin{rem}\label{graphgeneralization}
Let $G$ be
a simple graph, and let $\overset{\leftrightarrow}{G}$ be the graph obtained by replacing every edge in $E(G)$ with double arcs. Since a vertex $u$ is a neighbor of vertex $v$ in $G$ if and only if $u$ is an out-neighbor of $v$ in $\overset{\leftrightarrow}{G}$, it follows that $\throt(G) = \throt(\overset{\leftrightarrow}{G})$.
\end{rem}

\subsection{Monotonicity and characterizations}
\label{subsec:monochar}

Recall that a graph parameter $p$ is subgraph monotone if $p(H) \leq p(G)$ whenever $H$ is a subgraph of $G$. Minor and induced subgraph monotonicity are similarly defined. In \cite{Carlson19}, it is shown that the throttling number of undirected graphs is not subgraph monotone; therefore, it is not minor monotone. However, this result does not address whether the throttling number is induced subgraph monotone. Note that subgraphs, minors, and monotonicity for digraphs are defined analogously to undirected graphs. The following example illustrates that the throttling number is not induced subgraph monotone for oriented graphs (and therefore, directed graphs) and undirected graphs.

\begin{ex} \label{notIndSubMon}
Consider the oriented graph $\vec{H}$ on the right of Figure \ref{goldfish} as an induced subgraph of $\vec{G}$, shown on the left. A zero forcing set $B \subseteq V(\vec{G})$ with $|B|+ \pt(\vec{G};B) \leq 2 + 2 = 4$ is shown in blue. By checking all possible zero forcing sets of $\vec{H}$, we see that $\throt(\vec{H}) = 5$ which can be achieved using the set of blue vertices shown on the right. 

Next, consider the undirected graph $H$ on the right of Figure \ref{starinduced} as an induced subgraph of $G$, shown on the left. We can similarly observe that $\throt(G) \leq 4$ and $\throt(H) = 5$. Therefore, the throttling number is not induced subgraph monotone for oriented graphs and undirected graphs.
\end{ex}
\begin{figure}[H]
    \centering
    \includegraphics[scale=0.4]{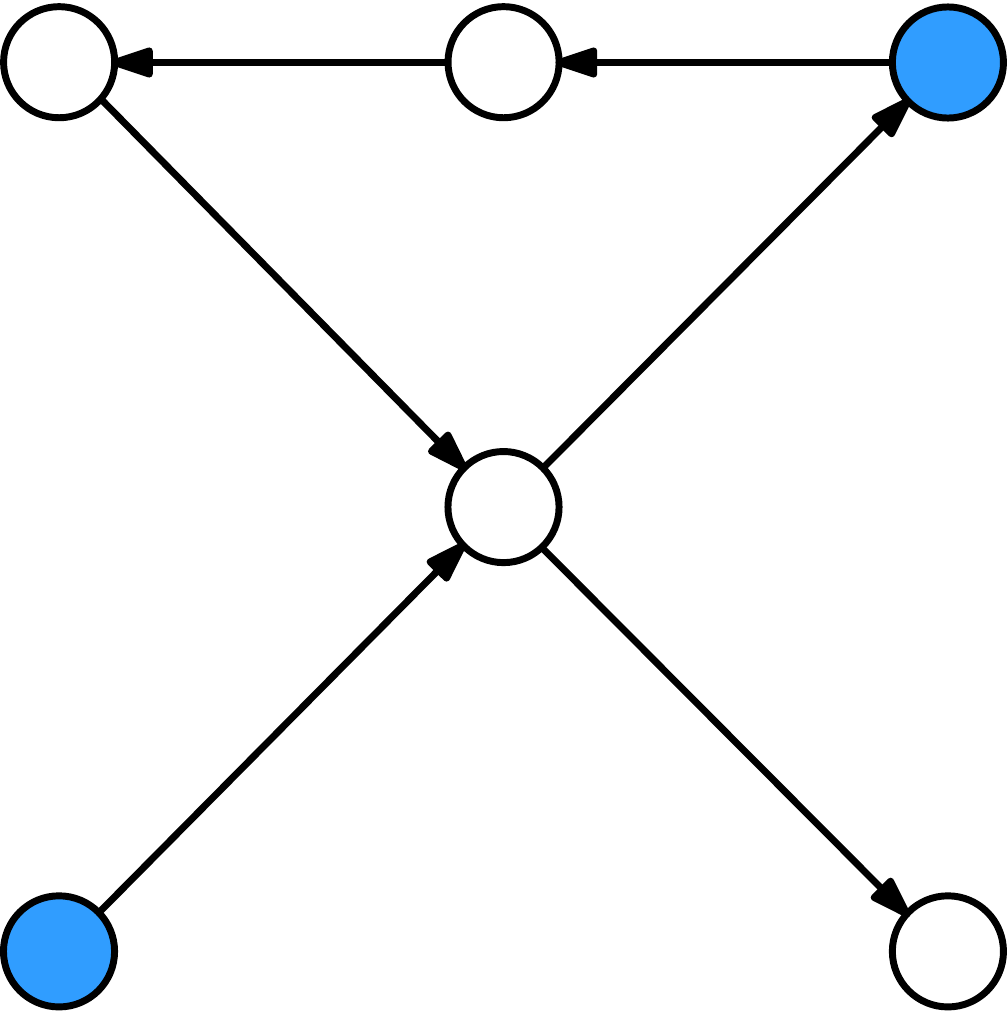}
    \hspace{20mm}
    \includegraphics[scale=0.4]{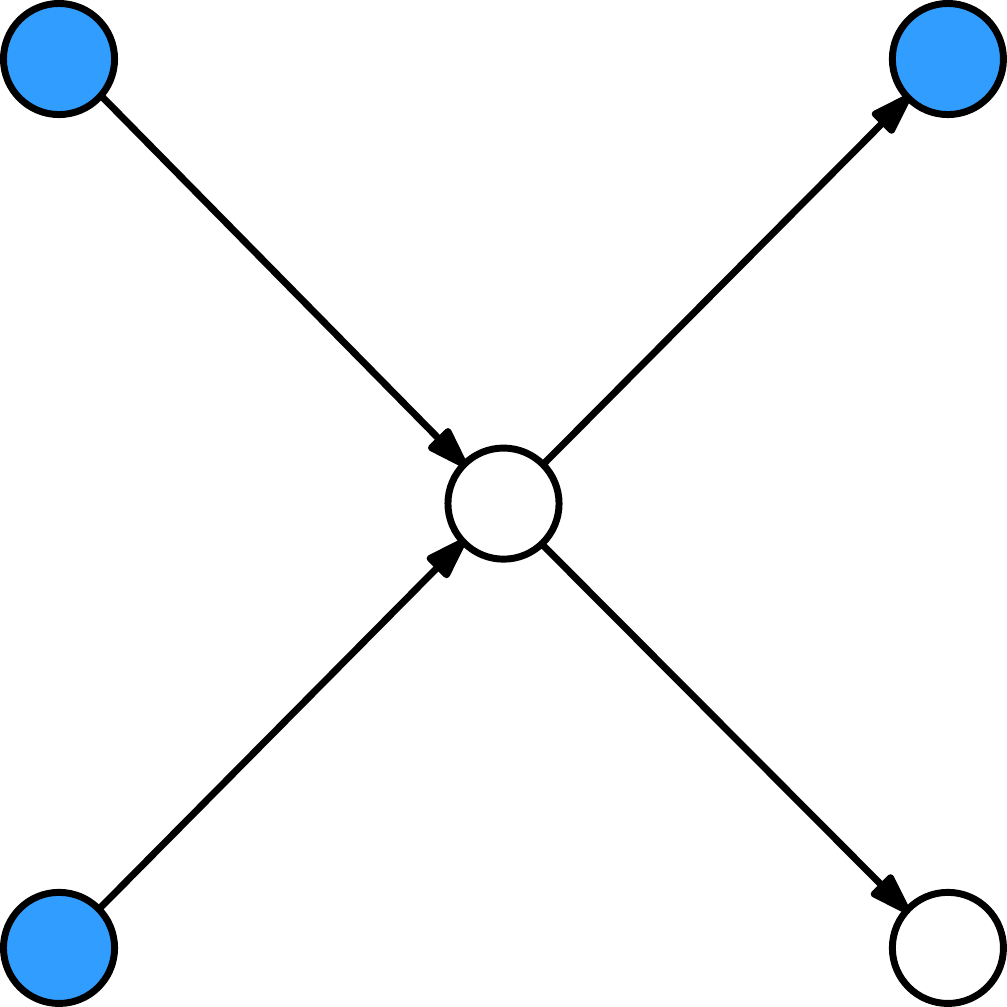}
    \caption{An oriented graph $\vec{G}$ with $\throt(\vec{G}) \leq 4$ is shown on left and an induced subgraph $\vec{H}$ of $\vec{G}$ is shown on the right with $\throt(\vec{H}) = 5$.} \label{goldfish}
\end{figure}
\begin{figure}[H]
    \centering
    \includegraphics[scale = 0.4]{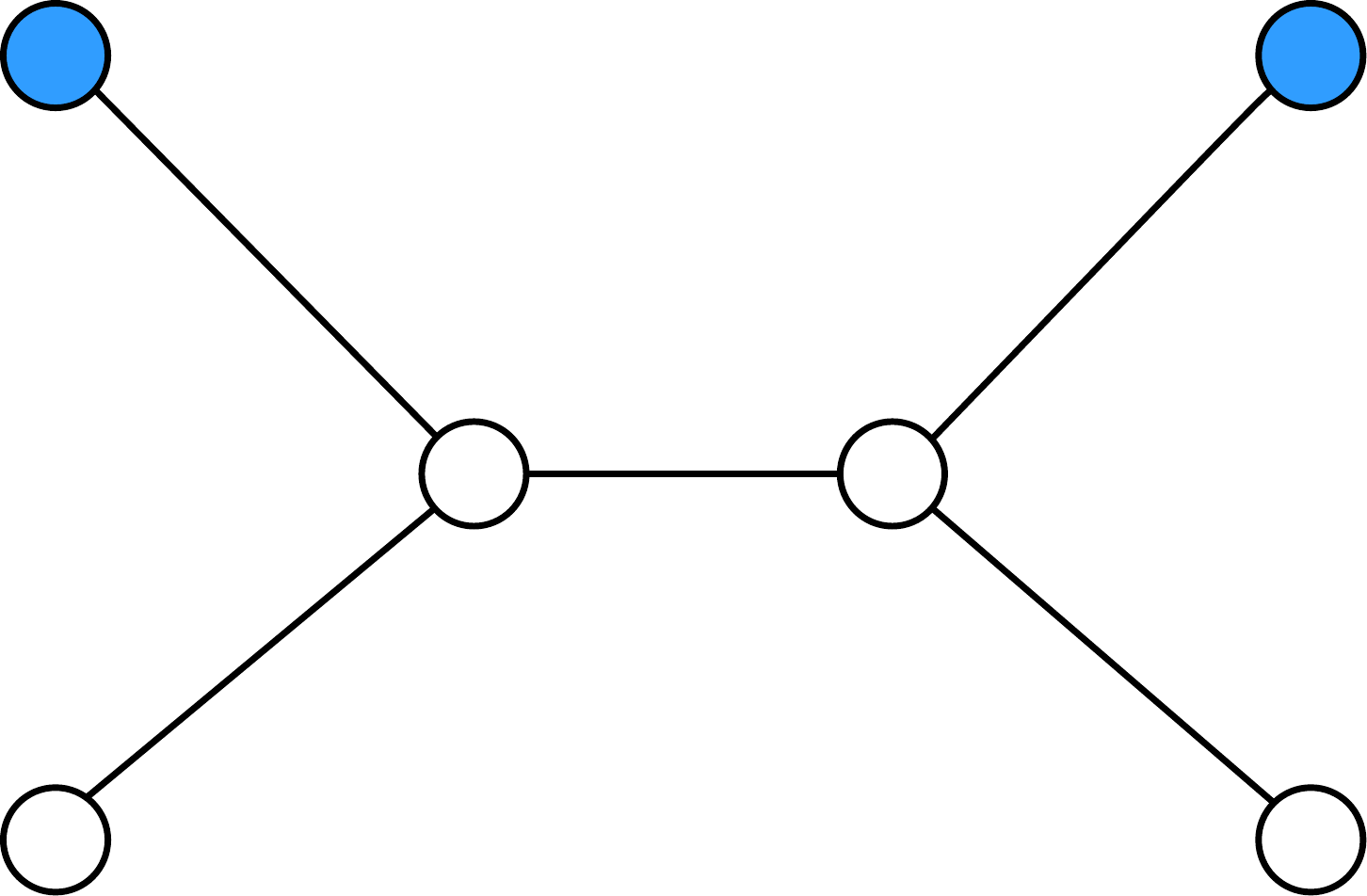}
    \hspace{1 cm}
    \includegraphics[scale = 0.4]{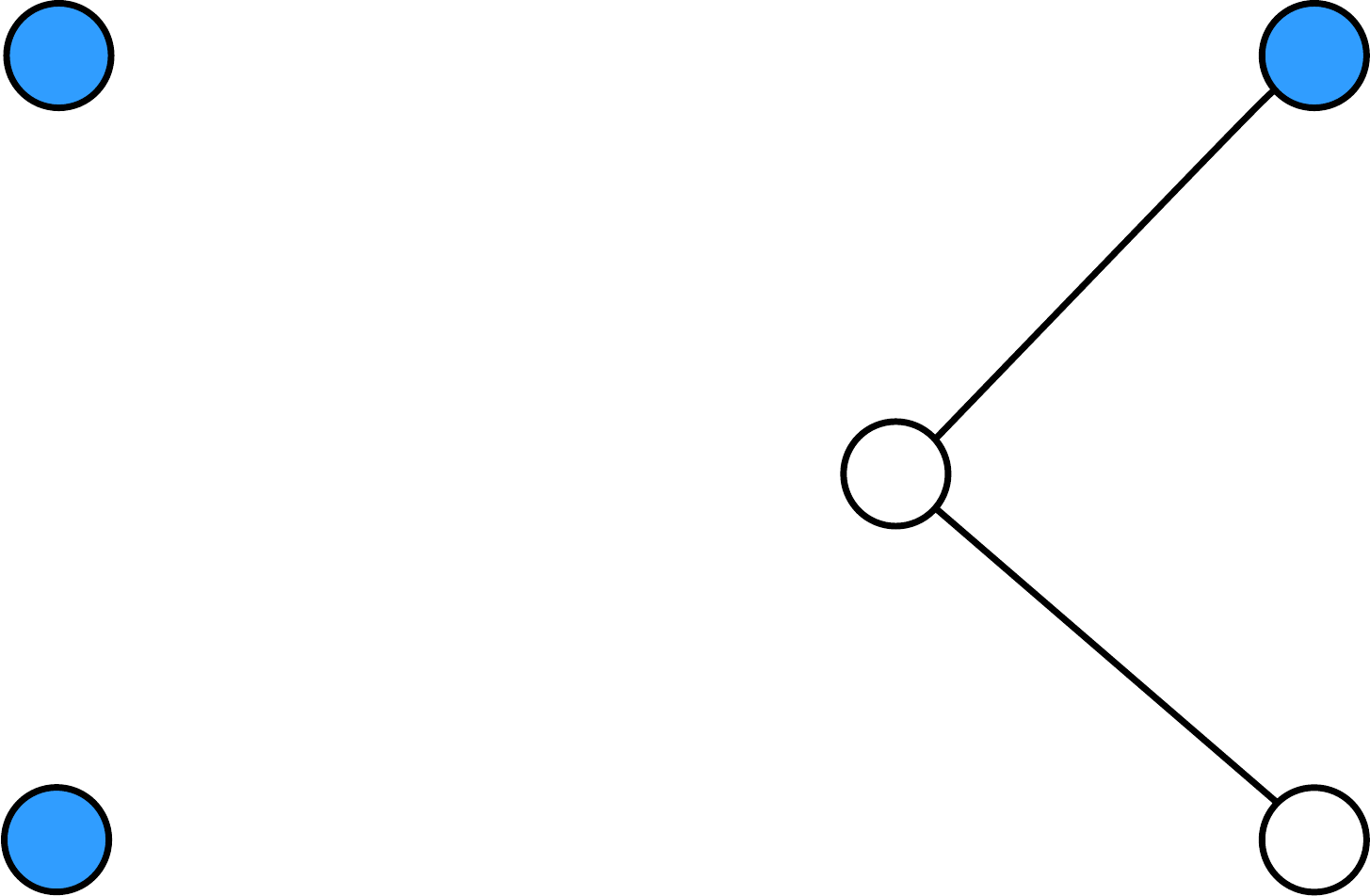}
    \caption{An undirected graph $G$ with $\throt(G) \leq 4$ is shown on left and an induced subgraph $H$ of $G$ is shown on the right with $\throt(H) = 5$. }
    \label{starinduced}
\end{figure}

Example \ref{notIndSubMon} highlights the fact that it is useful to know whether a digraph has throttling number at most $t$ for a given integer $t \geq 1$. A characterization for undirected graphs with this property is given in \cite[Theorem 4.1]{Carlson19}. With some modifications, an analogous characterization holds for directed graphs. To show this result, we need a digraph version of the important graphs utilized in \cite{Carlson19}.  We begin by providing a method for extending a digraph $\D$ into a major of $\D$ using a given a zero forcing set $B$, and a set of standard forces of $B$.

This construction requires defining the following graph. A \emph{Hessenberg path} with vertices $\{ v_1, v_2, \ldots, v_n\}$ is a simple digraph that contains all arcs of the form $(v_i, v_{i+1})$ for each $1 \leq i \leq n-1$ and does not contain any arc of the form $(v_i, v_j)$ with $j > i+1$. No restrictions are placed on \emph{back arcs}, i.e., arcs of the form $(v_i, v_j)$ with $i > j$. Note that a single isolated vertex is also a Hessenberg path. We also need some useful definitions from \cite{Berliner et al17} and \cite{proptime}. Given a simple digraph $\D$, a zero forcing set $B \subseteq V(\D)$, and a set of forces $\F$ of $B$, a sequence of vertices $(v_1, v_2, \ldots, v_k) \in V(\D)$ is a \emph{forcing chain} of $\F$ if $(v_i \to v_{i+1}) \in \F$ for each integer $1 \leq i \leq k-1$. A forcing chain of $\F$ is \emph{maximal} if it is not a subsequence of a larger forcing chain of $\calf$. 

\begin{defn}\label{extensiondef} Let $\D$ be a simple digraph and $B \subseteq V(\D)$ be a standard zero forcing set of $\D$. Suppose $\F$ is a set of forces of $B$ with $\pt(\D; B) = \pt(\D;\F)$. Let $\vec{\mathcal{H}}_1, \vec{\mathcal{H}}_2, \ldots, \vec{\mathcal{H}}_{|B|}$ be the induced Hessenberg paths in $\D$ formed by the maximal forcing chains of $\F$. For each vertex $v \in V(\D)$, let $\tau(v)$ be the number of time steps in the propagation process of $\F$ in which $v$ is blue and has not yet performed a force. Define the \emph{extension of $\D$ with respect to $B$ and $\F$}, denoted $\vec{\E}(\D, B, \F)$, to be the digraph created by the following construction.

First, for each Hessenberg path $\vec{\mathcal{H}}_i \in \D$, we construct a new Hessenberg path $\vec{\mathcal{H}}_i'$ so that for each $v\in \vec{\mathcal{H}}_i$, there are $\tau(v)$ copies of $v$ in $\vec{\mathcal{H}}_i'$, and for each pair of vertices $a,b \in \vec{\mathcal{H}}_i$ such that $a$ is forced before $b$ using $\calf$, every copy of $a$ is to the left of every copy of $b$ in $\vec{\mathcal{H}}_i'$. Add an arc going left to right between each pair of consecutive vertices in each $\vec{\mathcal{H}}_i$, creating a forward-directed path. We call these arcs \emph{path arcs}. Also, add the same back arcs of the form $(v,u) \in E(\vec{\mathcal{H}}_i)$ to $\vec{\mathcal{H}}_i'$ by connecting the first instance of $v$ to the first instance of $u$ in $\vec{\mathcal{H}}_i'$. Observe that $|V(\vec{\mathcal{H}}_i')| = \pt(\D;B) + 1$ for each $1 \leq i \leq |B|$, and the Hessenberg paths $\{\vec{\mathcal{H}}_1',\vec{\mathcal{H}}_2',\ldots, \vec{\mathcal{H}}_{|B|}'\}$ can be arranged into a $|B| \times (\pt(\D;B) + 1)$ array of vertices.
    
Then, for each arc \[(u,v) \in E(\D) \setminus \bigcup_{i=1}^{|B|} E(\vec{\mathcal{H}}_i), \] $v$ must be blue before $u$ can perform a force in $\D$ since $u$ and $v$ are in distinct Hessenberg paths. Therefore, there must be a copy $v'$ of $v$ and a copy $u'$ of $u$ in the $|B| \times (\pt(\D;B) + 1)$ array such that $v'$ appears in either the same column as $u'$ or in some column left of $u'$. For each of these arcs $(u, v)$, create an arc from the last instance of $u$ to the first instance of $v$ in each of their respective paths. Note that this will always create either a vertical arc or a backward arc, but never a forward arc. An illustration of this extension can be found in Figure \ref{extension}.
\end{defn}

\begin{figure}[H]
    \centering
    \includegraphics[scale = 0.4]{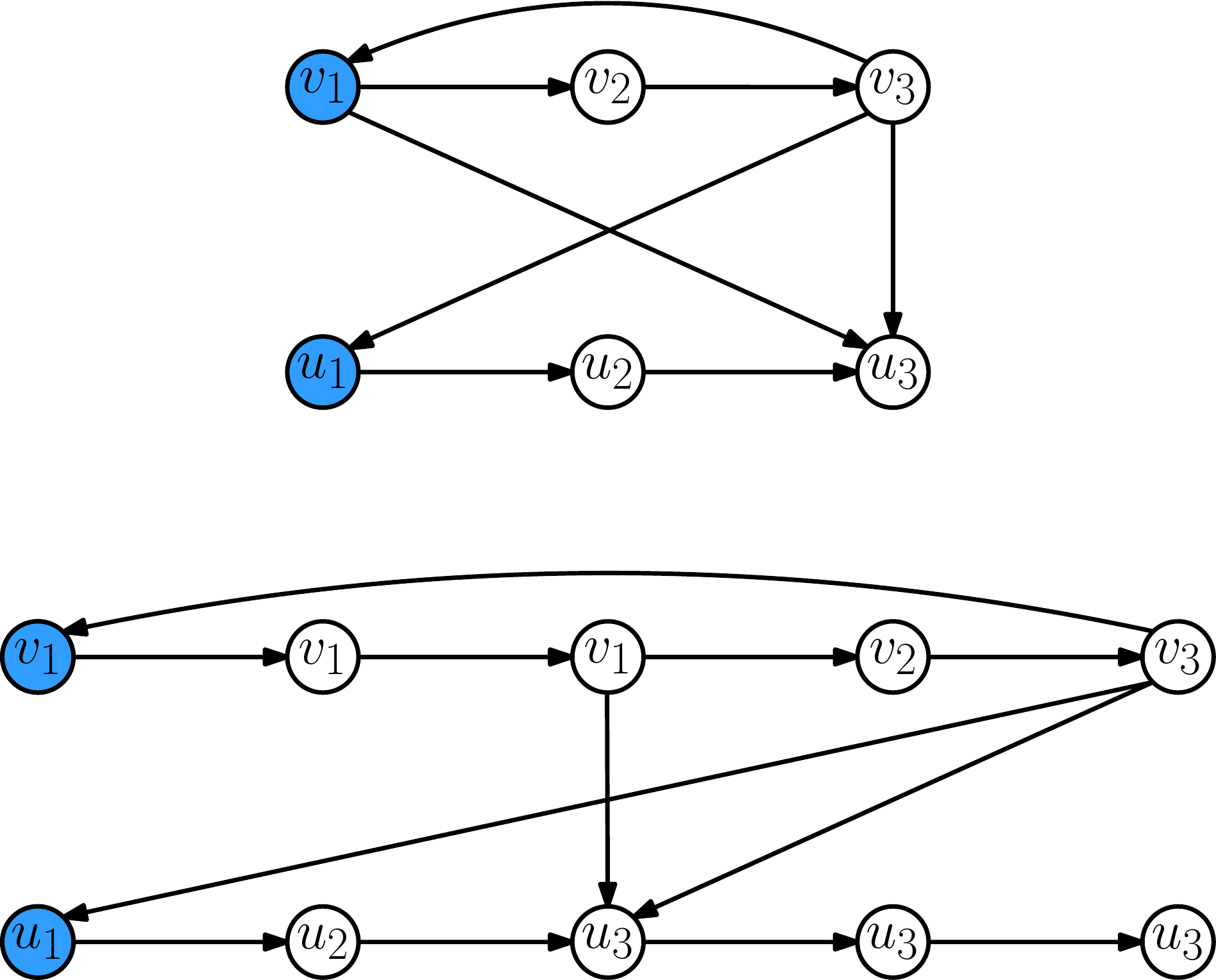}
    \caption{The digraph (above) has the following extension (below) by Definition \ref{extensiondef}.}
    \label{extension}
\end{figure}
It is important to note that every digraph $\D$ is a minor of any of its own extensions. Specifically, we can always contract the path arcs between copies of the same vertex to obtain the original digraph $\D$. Next, we construct a digraph, illustrated in Figure \ref{host}, that can be used to characterize graphs with a given throttling number. 

% Now, we construct a directed graph that serves the same purpose as  in Theorem 4.1 of \cite{Carlson19}.

\begin{defn}
For any integers $a \geq 1$ and $b \geq 0$, the digraph \emph{$H_{a,b + 1}$} is constructed via the following process. Begin with an undirected complete graph on $a \times (b+1)$ vertices and replace each edge with double arcs. Arrange the vertices in an array with $a$ rows and $b+1$ columns. Then, label every vertex with respect to its location on the array so that a vertex that lies in the  $i$-th row and $j$-th column is labeled as $v_{i,j}$, where  $0 \leq i \leq a-1$ and $0 \leq j \leq b$ (with $v_{0,0}$ and $v_{a-1,b}$ as the bottom-left and top-right corners respectively). Next, delete all the forward diagonal arcs, i.e., the arcs of the form $(v_{h,k}, v_{l,m})$ where $h \neq l$ and $m > k$. Also, delete all forward arcs of the form $(v_{h,k}, v_{l,m})$ where $h=l$ and $m > k + 1$ so that each row is an induced Hessenberg path with all possible backward arcs.

We define the \emph{path arcs} of $H_{a,b + 1}$ to be the arcs of the form $(v_{i,j},v_{i,j+1})$ for any $0 \leq i \leq a-1$ and $0 \leq j \leq b-1$. We refer to all other arcs in $H_{a,b+1}$ as \emph{non-path arcs}.
\end{defn}

\begin{figure}[H]
    \centering
    \includegraphics[scale=0.4]{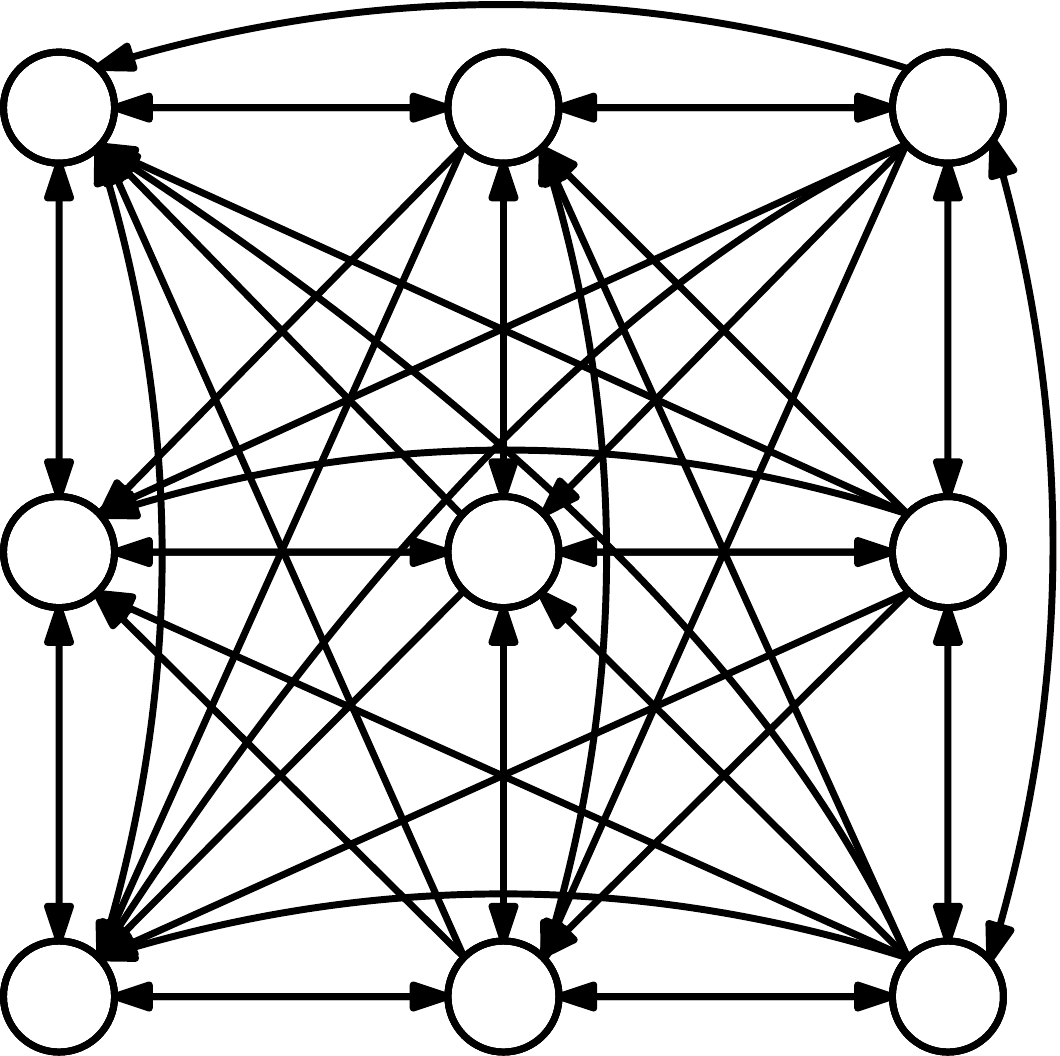}
    \caption{The graph $H_{3,3}$ is shown.}
    \label{host}
\end{figure}

To use a similar argument as in the proof of \cite[Theorem 4.1]{Carlson19}, we must ensure that we can contract an arc used to perform a force in a directed graph without increasing the throttling number. The following remark addresses this necessary condition. 

\begin{rem}\label{contractarc}
    Lemma 3.15 in \cite{Carlson19} states that in an undirected graph, contracting an arc in a forcing chain cannot increase the throttling number. Using this exact argument, it follows that contracting an arc that is used to perform a force also does not increase the throttling number on a directed graph.
\end{rem}

Using the graph $H_{a,b+1}$, we can now give an analogous theorem to \cite[Theorem 4.1]{Carlson19} for simple digraphs. 

\begin{thm}\label{charthm}
Given a simple digraph $\D$ and a positive integer $t$, $\throt (\D) \leq t$ if and only if there exist integers $a\geq 1$ and $b\geq 0$ such that $a+b = t$ and $\D$ can be obtained from $H_{a,b+1}$ by contracting path arcs and deleting non-path arcs.
\end{thm}

\begin{proof}
First, suppose $\throt(\D) \leq t$. Let $B \subseteq V(\D)$ be a zero forcing set of $\D$ that satisfies $\throt(\D;B) \leq t$ and let $\F$ be a set of forces of $B$ in $\D$ such that $\pt(\D;\F)=\pt(\D;B)$. Let $a=|B|$, $b'=\pt(\D;B)=\throt(\D;B)-a$, and $b=t-a$. Then, $b'\leq b$ and $\D$ is a minor of $\vec{\E}(\D, B, \F)$. Furthermore, $\vec{\E}(\D, B, \F)$ is a subdigraph of $H_{a, b'+1}$ which is a subdigraph of $H_{a, b+1}$. 
Note that by the construction of $\vec{\E}(\D, B, \F)$ and $H_{a,b+1}$, $\D$ can be obtained from $H_{a,b+1}$ by contracting path edges and deleting non-path edges.

Conversely, suppose $\D' = H_{a, b+1}$ with $a + b = t$ and $\D$ can be obtained from $\D'$ by contracting path arcs and deleting non-path arcs. Observe that the vertices in the left column of $H_{a,b+1}$ form a zero forcing set of size $a$ with propagation time $b$. By Remark \ref{contractarc}, contracting path arcs in $H_{a,b+1}$ does not increase the throttling number. Furthermore, the only arcs used to perform forces in $H_{a,b+1}$ are path arcs. Therefore, deleting non-path arcs also does not increase the throttling number. Thus, $\throt(\D) \leq \throt(H_{a,b+1}) \leq a + b$.
\end{proof}

Note that for a fixed integer $t \geq 1$, there are finitely many digraphs of the form $H_{a,b+1}$ with $a +b = t$, which means there are finitely many digraphs that can be obtained from them. Corollary \ref{cor:finite} follows from this observation.

\begin{cor}\label{cor:finite}
    If $t$ is a fixed positive integer, then there are finitely many digraphs $\D$ with throttling number equal to $t$.
\end{cor}
% To characterize $\throt (\D) = 1$, we must let $a=1$ and $b=0$, which just gives us a lone vertex. There are no arcs for us to contract or delete, so this is the only digraph with throttling number 1. To characterize $\throt (\D) = 2$, we can either set $a=2$ and $b=0$ or $a=1$ and $b=1$. These constructions give us the following options for $\D$: \begin{figure}[H]
%     \centering
%     \includegraphics[width = 13 cm]{char2.jpg}
% \end{figure} To characterize $\throt (\D) = 3$, we can set $a=3$ and $b=0$, $a=2$ and $b=1$, or $a=1$ and $b=2$. This gives us the following options: \begin{figure}[H]
%     \centering
%     \includegraphics[width = 13 cm]{char3.jpg}
% \end{figure}

% \emilyc{I feel like going any further with this gets overly complicated. Also, there's significantly more than 30 total.}

\subsection{Flipping arcs and other digraph operations}\label{subsec:otherops}

Next, we study the effect of flipping arcs on throttling. We start by giving a relationship between the throttling numbers of two directed graphs $\D$ and $\D_0$, where $\D_0$ can be obtained from $\D$ by flipping a single arc, i.e., replacing $(a,b)\in E(\D)$ with $(b,a)$.

\begin{prop}\label{singleflip}
    Flipping an arc $(a,b) \in E(\D)$ of a simple digraph $\D$, where $(b,a) \notin E(\D)$, to achieve a new graph $\D_0$ cannot increase the throttling number of the graph by more than one.
\end{prop}

\begin{proof}
Let $B$ be a zero forcing set of $\D$ with a set of forces $\F$ such that $\pt(\D; \F) = \pt(\D; B)$ and $\throt(\D) = \throt(\D; B)$. Suppose $\D_0$ is the digraph obtained from $\D$ by flipping an arbitrary arc $(a,b) \in E(\D)$ where $(b,a) \notin E(\D)$. To show that $\throt(\D_0) \leq \throt(\D) + 1$, it suffices to find a zero forcing set $B_0 \subseteq V(\D_0)$ with $|B_0|\leq |B| + 1$ and $\pt(\D_0;B_0) \leq \pt(\D; B)$. 

First, suppose $(a \to b) \in \F$.    
In this case, we claim that $B_0 = B \cup \{b\}$ is a zero forcing set of $\D_0$. Observe that $|B_0| \leq |B| + 1$ and $\F_0 = \F \setminus \{a \to b\}$ is a set of forces of $B_0$. To show that $\pt(\D_0; B_0) \leq \pt(\D; B)$, it is sufficient to show that the vertices forced by time step $t$ in  $\D$ are also forced in $\D_0$ by time step $t$. In other words, we aim to show for any time step $t$, we have $\F^{[t]} \subseteq \F_0^{[t]}$.
    
Let $t_B'$ be the time step in which the force $a \to b$ is performed in $\D$. Note that for any vertex $v \in V(\D)\setminus \{b\}$ and for all $0 \leq t \leq t_B'$, the set of white out-neighbors of $v$ in $\D_0$ at time step $t$ using $\F_0$ is a subset of the set of white out-neighbors of $v$ in $\D$ at time step $t$ using $\F$. Symbolically, this means $\big( N_{\D_0}^+(v) \setminus \F_0^{[t]} \big) \subseteq \big( N_{\D}^+(v) \setminus \F^{[t]} \big)$. Since $b$ is not yet blue at any time $t$ for $0 \leq t \leq t_B'$ in $\D$, it follows that $\F^{[t]} \subseteq \F_0^{[t]}$ for all $0 \leq t \leq t_B'$.
    
Observe that $b$ is the only vertex that gained an out-neighbor in $\D_0$, namely $a$, after flipping the arc in $\D$. At time $t_B'$, the vertex $a$ must be blue in $\D$ using $\F$ in order to force $b$. Therefore, vertex $a$ is also blue at time $ t_B'$ in $\D_0$ using $\F_0$. Since $a \in \calf\upc{t_B'}$ and $\F^{[t_B']} \subseteq \F_0^{[t_B']}$,  $a \in \F_0^{[t_B']}$. This means that $\big(N_{\D_0}^+(b) \setminus \F_0^{[t_B']} \big) \subseteq \big(N_{\D}^+(b) \setminus \F^{[t_B']} \big)$. Thus, $\big(N_{\D_0}^+(v) \setminus \F_0^{[t_B']} \big) \subseteq \big(N_{\D}^+(v) \setminus \F^{[t_B']} \big)$ for all $v \in V(\D)$. This includes the vertex $b$, which implies that for all $v \in V(\D)$ at time $t > t_B'$, $ \big(N_{\D_0}^+(v) \setminus \F_0^{[t]} \big) \subseteq \big(N_{\D}^+(v) \setminus \F^{[t]} \big)$. Hence, $\F^{[t]} \subseteq \F_0^{[t]}$ for all $t \geq 0$ and $\pt(\D_0;\F_0) \leq \pt(\D; \F)$.
    
Next, suppose $(a \to b) \notin \F$. In this case, we claim $B_0 = B \cup \{a\}$ is a zero forcing set of $\D_0$ with $|B_0| \leq |B| + 1$. If there exists a vertex $x$ such that $(x \to a) \in \F$, we let $\F_0 = \F \setminus \{x \to a\}$; otherwise, let $\F_0 = \F$. Observe that $b$ is the only vertex that gained an out-neighbor after flipping $(a,b)$, namely $a$, which is blue from the start. Hence, for all $v\in V(\D)$, we have $\big( N_{\D_0}^+(v) \setminus \F_0^{[t]} \big) \subseteq \big( N_{\D}^+(v) \setminus \F^{[t]} \big)$, which implies $\F^{[t]} \subseteq \F_0^{[t]}$ for all $t \geq 0$. Thus, $\pt(\D_0;\F) \leq \pt(\D; \F)$.

Note that $\pt(\D_0;\F_0) \leq \pt(\D; \F)$ and $|B_0|\leq |B| + 1$ in both cases. Therefore, \[ \throt(\D_0) \leq \throt(\D_0; B_0) \leq |B_0| + \pt (\D_0; \F_0) \leq |B| +1 + \pt(\D; \F) = \throt(\D)+1. \qedhere \]
\end{proof}

\begin{cor}\label{arcflip}
    If a simple digraph $\D_0$ is obtained from another simple digraph $\D$ by flipping a single arc $(a,b) \in E(\D)$ where $(b,a) \notin E(\D)$, then $|\throt(\D_0)-\throt(\D)|\leq 1$.
\end{cor}

\begin{proof}
    By Proposition \ref{singleflip}, the throttling number cannot increase by more than 1, so it suffices to prove that the throttling number cannot decrease by more than 1. Suppose that $\D_0$ can be obtained from $\D$ by flipping a single arc and that the throttling number decreases by more than 1. In turn, $\D$ can be obtained from $\D_0$ by flipping a single arc, and the throttling number increases by more than 1. This contradicts Proposition \ref{singleflip}.
\end{proof}

Corollary \ref{arcflip} motivates further study of the throttling number as opposed to the propagation time of a digraph $\D$, as $\throt(\D)$ behaves more predictably than $\pt(\D)$ does when a single arc is flipped. For example, a path with 4 vertices with all arcs going in one direction has propagation time 3. However, if the arc incident to the source is flipped, the propagation time becomes 1, so an arc flip can change this parameter by more than 1.

This raises a question: how do propagation time and throttling number change when every arc in the digraph is reversed? The \emph{transpose} of a digraph $\D$, denoted $\D^T$, is obtained by flipping all of its arcs. Additionally, the \emph{terminus of $\F$}, denoted $\term(\F)$, is the set of vertices that do not perform a force in $\F$.  The \emph{reversal of a set of forces $\F$,} denoted $\rev(\F)$, is the set of forces $\F$ found by reversing the direction of each arc in $\F$. Observe that $\term(\F)$ is a zero forcing set of $\D^T$ with $\rev(\F)$ as a set of forces. 

The next result, in \cite{Berliner et al17}, relates the propagation time of a set of forces in a digraph to that of its reversal in the digraph's transpose.

\begin{lem}\cite[Corollary 2.4]{Berliner et al17}\label{reverseprop}
Let $\D = (V,E)$ be a simple digraph, $B \subseteq V(\D)$ be a minimum zero forcing set of $\D$, and $\F$ be a set of forces of $B$ such that $\pt(\D;\F) = \pt(\D; B)$. Then, $\pt(\D, \F) = \pt(\D^T;\rev(\F))$.
\end{lem}

Although Lemma \ref{reverseprop} uses a minimum zero forcing set $B$, an identical proof yields the same result when $B$ is not minimum. We now consider whether $\calf$ and $\rev(\calf)$ can be replaced by their respective zero forcing sets (namely, $B$ and $\term(\calf)$) in the equation $\pt(\D, \F) = \pt(\D^T;\rev(\F))$. Figure \ref{revNotWorking} illustrates that this is not always the case.

\begin{figure}[H]
    \centering
    \includegraphics[scale=0.5]{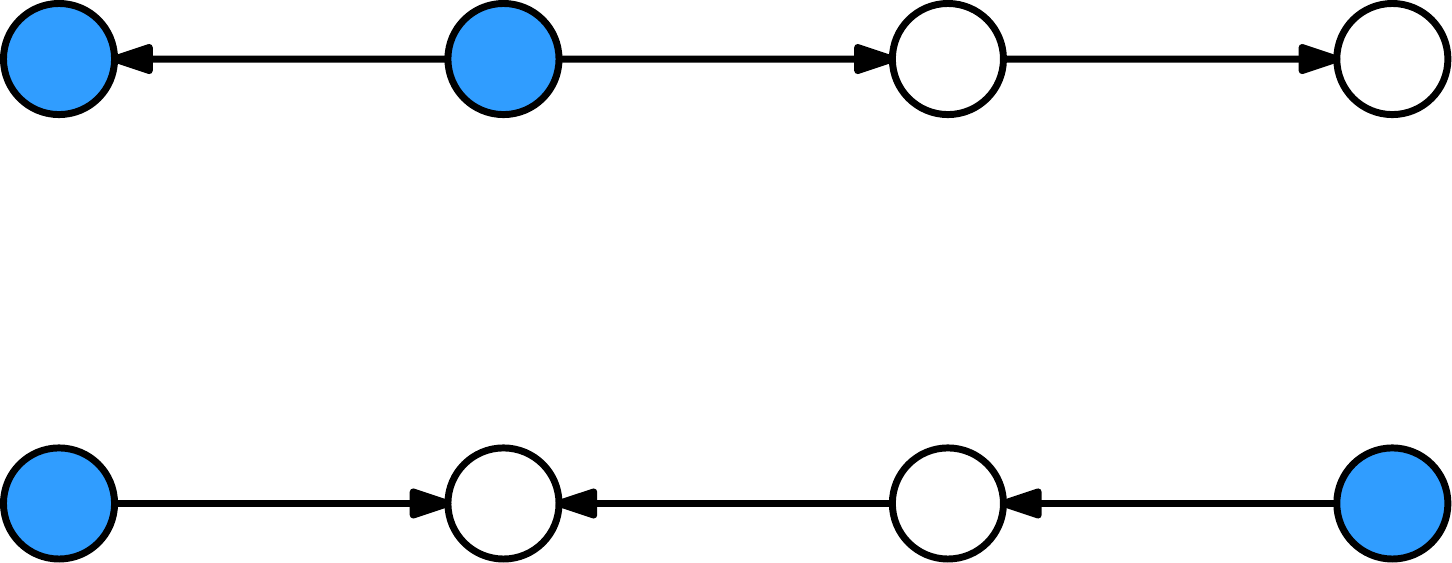}
    \caption{For the graph $\D$ (above), $B$ is the set of blue vertices satisfying $\pt(\D;B)=2$. For $\D^T$ (below), $\pt(\D^T;\term(\F))=1$, where $\F$ is the unique set of forces of $B$ in $\D$.}
    \label{revNotWorking}
\end{figure}

It is easy to see for any digraph $\D$, zero forcing set $B \subseteq V(\D)$, and set of forces $\calf$ of $B$, that $|\term(\calf)| = |B|$ (see \cite{Berliner et al17}). This fact, together with Lemma \ref{reverseprop}, can be used to show that $\pt(\D;B) = \pt(\D^T; \term(\calf))$ holds if $B$ is more carefully chosen. In particular, we obtain our desired equality if $B$ is chosen to have optimal propagation time for its size. This motivates the following definition which generalizes propagation time.

% This fact along with the following Remark is used to characterize the relationship between the throttling number of a simple digraph and that of its transpose.

\begin{defn}
For a simple digraph $\D$ and integer $k \geq 0$, define the \emph{$k$-propagation time of $\D$} as $\pt_k(\D) = \min\{\pt(\D; B) ~|~ B \text{ is a zero forcing set and } |B| = k\}$. 
\end{defn}

\begin{prop}\label{terminus}
    Let $\D$ be a simple digraph, $B \subseteq V(\D)$ be a zero forcing set of $\D$ such that $|B|=k$ and $\pt(\D; B)=\pt_{k}(\D)$, and $\F$ be a set of forces of $B$ such that $\pt(\D;\F) = \pt(\D; B)$. Then, $\pt(\D; B)=\pt(\D^T; \term(\F))$.
\end{prop}

\begin{proof}
First, we will show that $\pt(\D^T; \rev(\F))= \pt(\D^T; \term(\F))$. Note that by definition, we already know that $\pt(\D^T;\rev(\F)) \geq \pt(\D^T;\term(\F))$. Suppose for the sake of contradiction that $\pt(\D^T; \rev(\F)) > \pt(\D^T; \term(\F))$. Then, there must exist some set of forces $\F'$ of $\term(\F)$ such that $\pt(\D^T; \F') < \pt(\D^T; \rev(\F))$. From Lemma \ref{reverseprop}, we know $\pt(\D^T;\rev(\F)) = \pt(\D; \F)=\pt(\D; B)$, so $\pt(\D^T; \F') < \pt(\D; B)$. However, Lemma \ref{reverseprop} also implies that $\pt(\D^T; \F')=\pt(\D; \rev(\F'))$, so \[\pt(\D; \rev(\F')) = \pt(\D^T; \F') < \pt(\D; B) = \pt_{k}(\D). \] 
We see that the propagation time of $\rev(\F')$ on $\D$ is strictly less than $\pt_{k}(\D)$. However, $\rev(\F')$ is a set of forces of $\term(\F')$ and $|\term(\F')|=k$, so this is a contradiction. Thus, $\pt(\D^T; \term(\F)) = \pt(\D^T; \rev(\F)) = \pt(\D;\F) = \pt(\D;B)$.
\end{proof}

Now, Proposition \ref{terminus} can be used to equate $\throt(\D)$ and $\throt(\D^T)$.

\begin{thm}\label{reversethrot}
For a simple digraph $\D$, 
$\throt(\D) = \throt(\D^T)$.
\end{thm}

\begin{proof}
Let $\D = (V(\D), E(\D))$ be a simple digraph and choose a zero forcing set $B$ and set of forces $\F$ of $B$ such that $\pt(\D;\F) = \pt(\D; B)$ and $\throt(\D) = |B| + \pt(\D;B)$. Therefore, $\pt(\D;B)=\pt_{k}(\D)$ for $k = |B|$; otherwise, $\throt(\D) < |B| + \pt(\D; B)$ which is a contradiction. Thus, $\pt(\D^T;\term(\F)) = \pt(\D; B)$ by Proposition \ref{terminus}. It follows from this fact and the definition of the throttling number that
\begin{eqnarray*}
    \throt(\D^T) &\leq& |\term(\F)| + \pt(\D^T; \term(\F))\\
    & = & |B| + \pt(\D; B)\\
    & = & \throt(\D). 
\end{eqnarray*}
Flipping the roles of $\D$ and $\D^T$, we obtain the reverse inequality $\throt(\D) = \throt \left((\D^T)^T \right) \leq \throt(\D^T)$. Therefore, $\throt(\D) = \throt(\D^T)$.
\end{proof}

The equality in Theorem \ref{reversethrot} can be extended to the case where only a single component of a digraph is transposed. To see this, we first need the following important lemma.

\begin{lem}\label{reverseptk}
    Let $\D$ be a simple digraph with $|V(\D)|=n$, and let $k$ be any integer such that $\Z(\D) \leq k \leq n$. Then, $\pt_{k}(\D)=\pt_{k}(\D^T)$.
\end{lem}

\begin{proof}
    Let $B$ be a zero forcing set of $\D$ and $\F$ be a set of forces of $B$ such that $|B|=k$ and $\pt(\D;\F)=\pt(\D;B)=\pt_k(\D)$. Then, $\pt_k(\D) = \pt_k(\D; B) = \pt_k(\D^T; \term(\F)) \geq \pt_k(\D^T)$ by Proposition \ref{terminus}. By reversing the roles of $\D$ and $\D^T$, we obtain the reverse inequality.
\end{proof}

\begin{thm}
    Let $\D_1$ and $\D_2$ be simple digraphs. Then, $\throt(\D_1 \cup \D_2)=\throt(\D_1^T \cup \D_2)$.
\end{thm}

\begin{proof}
    Let $\D = \D_1 \cup \D_2$ and $B$ be a zero forcing set of $\D$ such that $\throt(\D)=\throt(\D;B)$. Note that $\throt(\D;B)=|B_1|+|B_2|+\max \{\pt(\D_1; B_1), \pt(\D_2; B_2)\}$, where $B_1 = B \cap V(\D_1)$ and $B_2 = B \cap V(\D_2)$. Also, let $k_1=|B_1|$. There exists some $B_1' \subseteq V(\D_1^T)$ such that $|B_1'|=k_1$ and $\pt(\D_1^T;B_1')=\pt_{k_1}(\D_1^T)$ by definition of $\pt_{k_1}(\D_1^T)$. By Lemma  \ref{reverseptk}, it follows that $\pt(\D_1^T;B_1') = \pt_{k_1}(\D_1)$. Thus, $\pt(\D_1^T; B_1')=\pt_{k_1}(\D_1)\leq \pt(\D_1;B_1)$ by definition of $\pt_{k_1}(\D_1)$. Hence,
    \begin{eqnarray*}
        \throt(\D_1^T \cup \D_2) &\leq& \throt(\D_1^T \cup \D_2; B_1' \cup B_2)\\&=&|B_1'|+|B_2|+\max \{\pt(\D_1^T; B_1'), \pt(\D_2; B_2)\} \\&\leq& |B_1|+|B_2|+\max \{\pt(\D_1; B_1), \pt(\D_2; B_2)\} ~=~ \throt(\D_1 \cup \D_2). 
    \end{eqnarray*}
    By reversing the roles of $\D$ and $\D^T$, we obtain the reverse inequality.
\end{proof}

Another natural question that arises is whether the throttling number of a directed or undirected graph resulting from a disjoint union ($\cup$) of two graphs can be determined using the throttling numbers of each of the operands. Let $\D_1$ and $\D_2$ be directed or undirected graphs. If $B_1 \subseteq V(\D_1)$ and $B_2 \subseteq V(\D_2)$ are zero forcing sets such that $\throt(\D_1) = \throt(\D_1; B_1)$ and $\throt(\D_2) = \throt(\D_2; B_2)$, it is clear that \[ \throt{(\D_1\cup \D_2)} \leq |B_1| + |B_2| + \max\{\pt(\D_1; B_1), \pt(\D_2; B_2)\}. \]

%We can give an upper bound in terms of these values. 

% \begin{prop}
%     Let $\D_1$ and $\D_2$ be directed or undirected graphs. Let $B_1 \subseteq V(\D_1)$ be a set of vertices that realizes $\throt{(\D_1)}$, and $B_2 \subseteq V(\D_2)$ be a set of vertices that realizes $\throt{(\D_2)}$. Then, \[\throt{(\D_1\cup \D_2)} \leq |B_1| + |B_2| + \max\{\pt(\D_1; B_1), \pt(\D_2; B_2)\}. \]
% \end{prop}

% \emilyc{Can we find the other paper that did this? If we can, we'll turn this into a remark and delete the proof.}

% \begin{proof} In $\D_1 \cup \D_2$, color $B_1$ and $B_2$ blue. Then, $B_1$ will force all vertices in $\D_1$ in $\pt(\D_1; B_1)$ time steps and $B_2$ will force all vertices in $\D_2$ in $\pt(\D_2; B_2)$ time steps. Once the number of time steps exceeds $\max\{\pt(\D_1; B_1), \pt(\D_2;B_2)\}$, all vertices will have been forced. Thus, \[ \throt{(\D_1\cup \D_2)} \leq |B_1| + |B_2| + \max\{\pt(\D_1; B_1), \pt(\D_2; B_2)\}.\qedhere \]
% \end{proof}

However, the throttling number of a disjoint union of two digraphs is not necessarily equal to this upper bound, as shown in Example \ref{notUpperBound}.

\begin{ex} \label{notUpperBound}
Let $\vec{G}_1$ be a path on four vertices with all arcs oriented in the same direction and let $\vec{G}_2$ be the disjoint union of two copies of $\vec{G}_1$. By considering all zero forcing sets of $\vec{G}_1$ and $\vec{G}_2$ respectively, we see that $\throt(\vec{G}_1) = 3$ and $\throt(\vec{G}_2) = 5$. Zero forcing sets $B_1$ and $B_2$ that realize these respective throttling numbers are shown in blue on the left of Figure \ref{disjointpaths}. However, a zero forcing set $B \subseteq \vec{G}_1 \cup \vec{G}_2$ that realizes $\throt(\vec{G}_1 \cup \vec{G}_2; B) = 6$ is shown in blue on the right of Figure \ref{disjointpaths}. Therefore, \[ \throt(\vec{G}_1 \cup \vec{G}_2) \leq 6 < 2 + 2 + 3 = |B_1| + |B_2| + \max\{\pt(\vec{G}_1; B_1), \pt(\vec{G}_2; B_2) \}. \]

% Note that $\throt(\vec{G}_1) = 3$ if we initially color the source and the vertex adjacent to the sink blue so that all vertices are forced in one time step. 
% %For this graph and the rest in this example, see Figure \ref{disjointpaths} for illustrations.
% Moreover, $\throt(\vec{G}_2) = 5$ if we color both sources blue and all vertices are forced in three time steps. If the aforementioned bound were equality, we would have $\throt( \vec{G}_1 \cup \vec{G}_2) = 7$;
% %\[ \throt (\vec{G}_1 \cup \vec{G}_2) = 2 + 2 + 3 = 7. \]
% however, $\throt(\vec{G}_1 \cup \vec{G}_2) \leq 6 $ by coloring all three sources and fully forcing both directed graphs in $3$ time steps. Thus, the upper bound is not equality. 

\begin{figure}[H]
   \centering
   \begin{subfigure}[b]{0.4\textwidth}
       \includegraphics[width=\textwidth]{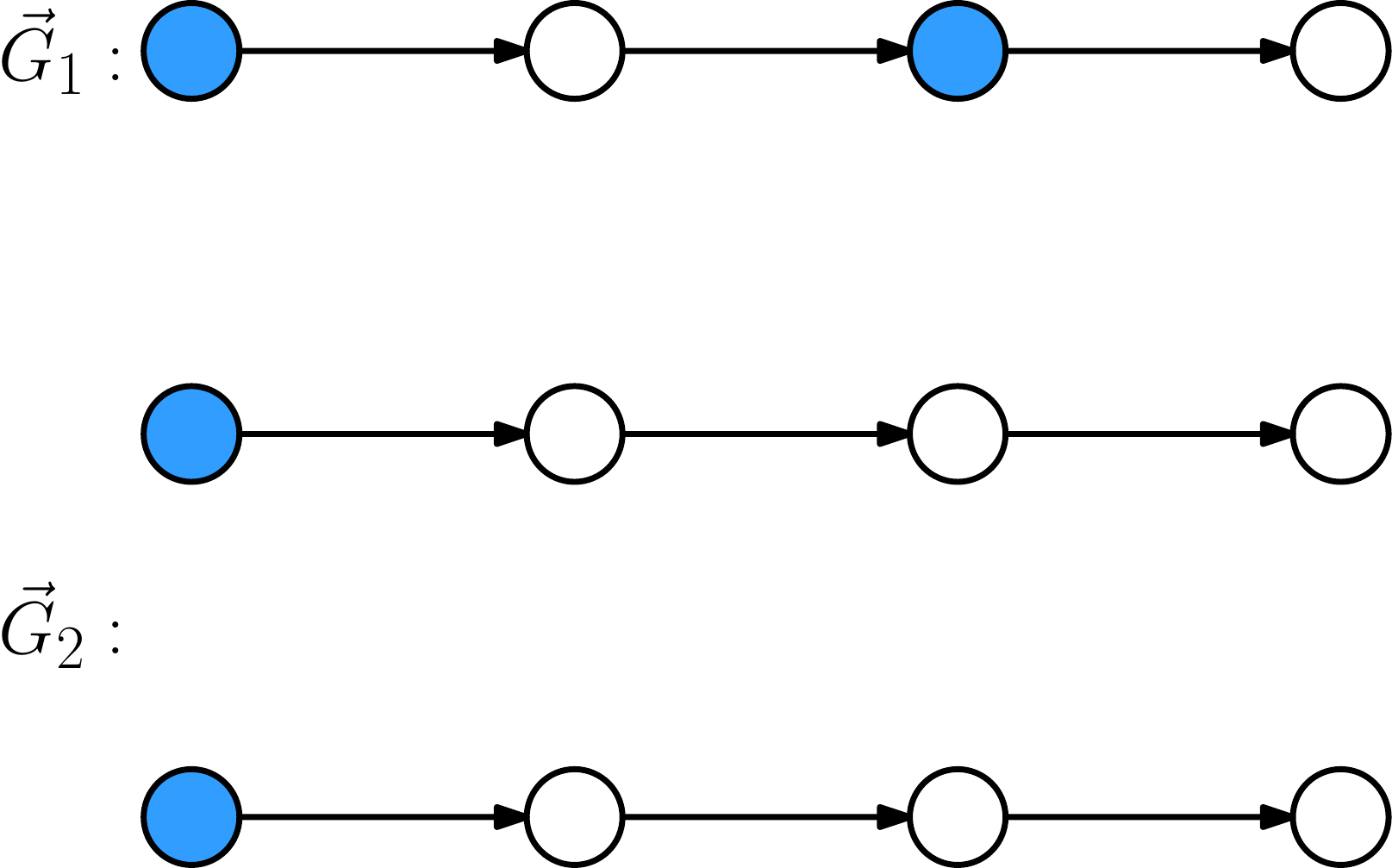}
   \end{subfigure}
   \hspace{1 cm}
   \begin{subfigure}[b]{0.45\textwidth}
       \includegraphics[width=\textwidth]{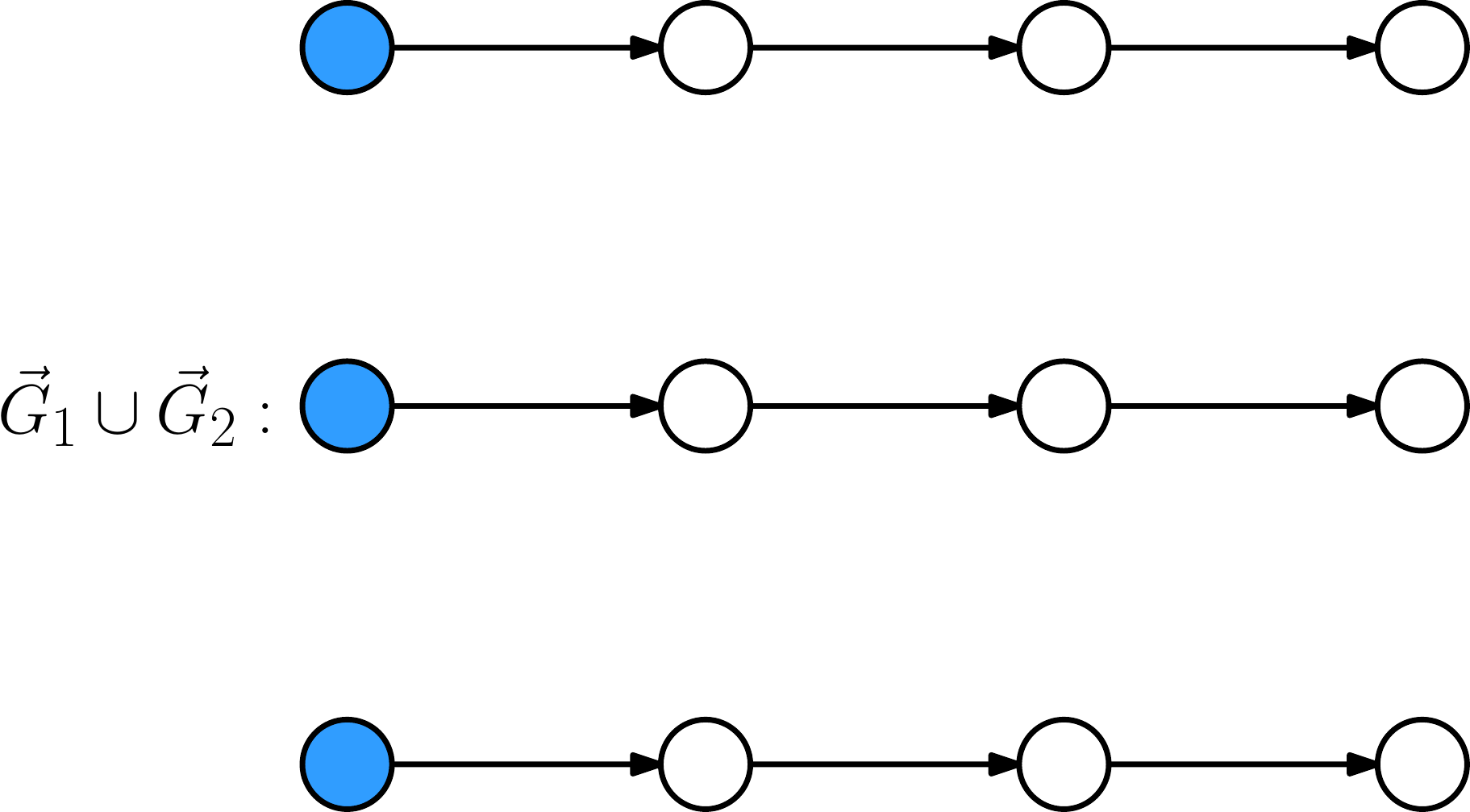}
    \end{subfigure}
\caption{Digraphs $\vec{G}_1$ and $\vec{G}_2$ are shown with $\throt(\vec{G}_1) = 3$, $\throt(\vec{G}_2) = 5$, and $\throt (\vec{G}_1 \cup \vec{G}_2) \leq 6$.} \label{disjointpaths}
\end{figure}
\end{ex}

% Additionally, there exists a pair of graphs such that the throttling process for their disjoint union uses neither of their individual throttling processes. Instead, it throttles both individual graphs sub-optimally in order the achieve the throttling number of the disjoint union. However, those graphs are too large for us to show here.

The final digraph operation we consider is adding or deleting a vertex in a digraph and its effect on the throttling number.

\begin{prop}
    Adding or deleting a vertex can change the throttling number of a simple digraph by at most one.
\end{prop}

\begin{proof}
    Let $\D$ be a simple digraph with a zero forcing set $B\subseteq V(\D)$ and set of forces $\F$ of $B$ such that $\throt(\D)=|B| + \pt(\D; \F)$.
    
    First, suppose a vertex $v$, along with a possibly empty set $A$ of arcs incident to $v$, is added to $\D$. Color $B$ $\cup$ $\{v\}$ blue. Then the same set of forces $\F$ can be used without increasing propagation time, since no vertices from the original digraph gain any white out-neighbors when $v$ and $A$ are added. Hence, throttling number can increase by at most one when a vertex is added, which implies that throttling number can decrease by at most one when deleting a vertex.
    
    Now, suppose instead that a vertex $v$ is deleted from $\D$ to obtain $\D'$. If $(u \rightarrow v) \in \F$ for some $u \in V(\D)$, then remove the force $u \rightarrow v$ from $\calf$. Also, if $(v \rightarrow w) \in \calf$ for some $w \in V(\D)$, remove the force $v \to w$ from $\F$ and add $w$ to $B$. Additionally, if $v \in B$, then remove $v$ from $B$. Call the resulting set of forces $\F'$ and set of vertices $B'$. Now color $B'$ blue in $\D'$. If the vertex $w$ exists, it can no longer be forced by $v$. However, in this case, $w$ is colored blue at time $t=0$. After deleting $v$, no vertices gain any white out-neighbors, so it is clear that $\pt(\D'; \F') \leq \pt(\D; \calf)$. Hence, the throttling number can increase by at most one when a vertex is deleted, which implies that throttling number can decrease by at most one when adding a vertex.
\end{proof}

%%%%%%%%%%%%%%%%%%%%%%%%%%%%%
\section{Throttling for orientations of simple graphs}\label{sec:throttlingorientations}

 For a given undirected graph, there are many ways to direct the edges and create an oriented graph. Naturally, the zero forcing number, propagation time, and throttling number vary among different orientations. These ranges of zero forcing numbers and propagation times are studied in \cite{ISUREU} and \cite{Berliner et al17} respectively. In this section, we investigate the range of throttling throttling numbers achieved by the orientations of a given simple graph. This idea is formalized in the following definition. 

\begin{defn}\label{def:OTI}
For a simple graph $G$, let $T = \{ \throt(\vec{G}) ~|~ \vec{G} \text{ is an orientation of } G\}$. The \emph{orientation throttling interval of $G$}, denoted $\oti(G)$, is the set of integers in the interval $[m,M]$ where $m$ and $M$ are the minimum and maximum values of $T$ respectively.
%$m = \min(T)$ and $M=\max(T)$. 
The graph $G$ is said to have a \emph{full} orientation throttling interval if $k \in T$ for each integer $k \in \oti(G)$.
\end{defn}

This new terminology naturally raises the question: which graphs, if any, have full orientation throttling intervals? We show that every simple graph has a full orientation throttling interval whose length is bounded in terms of the number of edges in the graph. We do so by stating an existing result from \cite{ISUREU} for general graph parameters and then applying this result to the throttling number using our findings in Section \ref{subsec:otherops}.

\begin{thm} \label{isureu}
\cite[Theorem 2.1]{ISUREU} Suppose $\beta$ is a positive-integer-valued digraph parameter with the following properties:
\begin{enumerate}
    \item $\beta(\vec{G}^T) = \beta(\vec{G})$
    
    \item
    If $(u,v) \in E(\vec{G})$ and $\vec{G}_0$ is obtained from $\vec{G}$ by replacing $(u,v)$ by $(v,u)$ (i.e. reversing the orientation of one arc), then $|\beta(\vec{G}_0)- \beta(\vec{G})| \leq 1$.
\end{enumerate}
Then for any two orientations $\vec{G}_1$ and $\vec{G}_2$ of the same graph $G$, $\beta(\vec{G}_2) - \beta( \vec{G}_1) \leq \floor{\frac{E(G)}{2}}$. Furthermore, every integer between $\beta(\vec{G}_2)$ and $\beta(\vec{G}_1)$ is attained as $\beta(\vec{G})$ for some orientation $\vec{G}$ of $G$.
\end{thm}

\begin{cor}\label{fullthrot}
     Let $G$ be a simple graph with orientation throttling interval $[m, M]$. Then, \[M-m \leq \floor{\frac{E(G)}{2}}.\] Furthermore, the orientation throttling interval of $G$ is full. 
\end{cor}

\begin{proof}
    Let $\vec{G}$ be any orientation of $G$ and $\vec{G}_0$ be any graph obtained from $\vec{G}$ by a single arc flip. Also, let $\vec{G}_m$ and $\vec{G}_M$ be orientations of $G$ such that $\throt(\vec{G}_m) = m$ and $\throt(\vec{G}_M) = M$. By Theorem \ref{reversethrot} and Corollary \ref{arcflip}, we have $\throt(\vec{G})=\throt(\vec{G}^T)$ and $|\throt(\vec{G})-\throt(\vec{G}_0)| \leq 1$. It follows from Theorem \ref{isureu} that $M-m \leq \floor{\frac{E(G)}{2}}$ and that $\oti(G)$ is full.
\end{proof}

Another question that the orientation throttling interval raises is whether the throttling number of the underlying simple graph is an element of its orientation throttling interval. In other words, is it true that for a graph $G$, there always exists some orientation $\vec{G}$ such that $\throt(\vec{G})=\throt(G)$? We show that this is the case for graphs $G$ that satisfy a constraint in terms of their independence number $\alpha(G)$. Recall that if $G$ is a graph, an independent set of $G$ is a subset of $V(G)$ that induces a subgraph of $G$ with no edges and $\alpha(G)$ is the size of a maximum independent set of $G$.

\begin{prop}\label{subset}
    Let $G$ be a simple graph with at least one edge and $\oti(G)=[m,M]$. Then, $m \leq \throt(G)$ and $\alpha (G) + 1 \leq M$. 
\end{prop}

\begin{proof}
    Let $B$ be a zero forcing set of $G$, and let $\F$ be a set of forces of $B$ that achieves $\throt(G)=|B|+\pt(G;\F)$. To prove that $ m \leq \throt(G)$, it suffices to show that there exists an orientation $\vec{G}$ of $G$ such that $\throt(\vec{G}) \leq \throt(G)$. We begin by orienting each edge in $G$ as follows. If $(u \to v) \in \F$, orient the edge between these vertices going from $u$ to $v$; otherwise, orient the edge in an arbitrary direction. Call the resulting oriented graph $\vec{G}$. Note that $\pt(\vec{G};\F) \leq \pt(G;\F)$ since the arcs in $\F$ exist in $\vec{G}$ and $N^+_{\vec{G}}(v) \subseteq N_G(v)$ for any vertex $v$. Thus,
    \[m \leq \throt(\vec{G}) \leq |B| + \pt(\vec{G};
    \F)  \leq |B| + \pt(G;\F) = \throt(G).\] %which implies $m \leq \throt(\vec{G}) \leq \throt(G)$.
    
    Now, we show that $\alpha (G) + 1 \leq M$ by constructing an orientation of $G$ that has throttling number $\alpha (G) +1$ or greater. To do so, make each vertex in a maximum independent set of $G$ a source. This is possible because no two vertices in an independent set have an edge between each other. Call some orientation that satisfies this condition $\vec{G}'$. Then, any zero forcing set of $\vec{G}'$ must include the $\alpha(G)$ sources. Further, since $G$ has at least one edge and it is impossible for both vertices incident to an edge to be sources, there is at least one vertex in $\vec{G}'$ that is not a source. Thus, $\throt(\vec{G}') > \alpha(G)$ and therefore $\alpha(G) + 1 \leq \throt(\vec{G}') \leq M$.
\end{proof}

\begin{cor} \label{subsetcor}
    Let $G$ be a simple graph with at least one edge. If $\throt (G) \leq \alpha(G) + 1$, then \[ [\throt (G), \alpha (G) +1] \subseteq \oti(G). \]
\end{cor}

Corollary \ref{subsetcor} gives an interval of integers that must be contained the OTI of a given graph if a certain inequality is satisfied. Another way to understand the orientation throttling interval is in terms of its maximum length. Proposition \ref{vertbounds} gives an interval of integers that always contains the OTI as a subset.

\begin{prop}\label{vertbounds}
    If $G$ is a simple graph on $n$ vertices, $\oti(G) \subseteq \left [ \ceil{2\sqrt{n} - 1}, n \right]$.
\end{prop}

\begin{proof}
    Given any oriented graph $\vec{G}$, coloring all vertices blue at the start yields throttling value $n$, so $\throt(\vec{G}) \leq n$. Also, for an undirected graph $G$, $\throt(G) \geq \ceil{2 \sqrt{n} - 1}$ from \cite{BY}. While this result is proven for undirected graphs, the same logic applies for oriented graphs since the argument does not rely on the edges being undirected.
\end{proof}

Propostion \ref{vertbounds} gives a lower and upper bound for the throttling number of any orientation of a graph in terms of the number of vertices. This is more useful than the length bound given by Corollary \ref{fullthrot} when the number of edges in a graph is large, as with complete graphs. In fact, the following proposition states that complete graphs can always be oriented to show that the bounds in Proposition \ref{vertbounds} are tight. For simplicity, if $G$ is a graph on $n$ vertices and $\oti(G) = \left [ \ceil{2\sqrt{n} - 1}, n \right]$, we say that the OTI of $G$ is \emph{maximum}. 

\begin{thm}\label{compmax}
The orientation throttling interval of a complete graph is maximum. 
\end{thm}

\begin{proof}
    To construct a tournament $\vec{K}_n$ for some $n$ with $\throt(\vec{K}_n)=n$, start with $n$ vertices labeled $v_0, v_1, \ldots, v_{n-1}$ and for each pair of distinct vertices $v_i$ and $v_j$ where $i > j$, add the arc $(v_i, v_j)$. 
%To construct a tournament $\vec{K}_n$ for some $n$ in such that $\throt(\vec{K}_n)=n$, start with $\vec{K}_1$ and label the vertex $v_0$. Now we iteratively obtain $\vec{K}_{m+1}$, where $1\leq m\leq n-1$, by adding a vertex labeled $v_{m}$ to $\vec{K}_m$ and adding an arc from $v_m$ to every vertex in $\vec{K}_m$. Repeat this process until there are $n$ vertices. Note that, for $0 \leq k \leq n-1$, every vertex labeled $v_k$ has out-degree $k$. Furthermore, the arc $(v_a,v_b)$ exists if and only if $a > b$.
This construction is illustrated with $n = 5$ in Figure \ref{completeMax}. Now, begin with any zero forcing set of $\vec{K_n}$ colored blue, and suppose $n \geq 4$, as the result is trivial for $n \leq 3$. Our goal is to prove that $\throt(\vec{K}_n)=n$, which is true if only one vertex can ever be forced at a single time step. For the sake of contradiction, suppose $v_a,v_b,v_x,v_y$ are distinct vertices such that $v_a \to v_x$ and $v_b \to v_y$ in time step $t$. Also, suppose without loss of generality that $b > a$. Then, by construction, $a > x$ and $b > y$, which implies $b > a > x$. Since $b > x$ and $b > y$, vertices $v_x$ and $v_y$ are both white out-neighbors of $v_b$ at time step $t$. Thus, $v_b$ cannot perform a force this time step, which is a contradiction. Hence, only one vertex can be forced at a single time step for any zero forcing set, so $\throt(\vec{K}_n)=n$.
    
    Now, to construct a tournament $\vec{K}_n$ for some $n$ such that $\throt(\vec{K}_n)=\ceil{2\sqrt{n}-1}$, let $m$ be the largest integer such that $m^2 \leq n$, and let $r = n - m^2$. Note that $r \leq 2m$; otherwise $n=m^2 + r \geq m^2 + 2m + 1 = (m+1)^2$, contradicting our choice of $m$. Now, choose the smallest graph of the form $H_{a,b+1}$ with $m$ columns that has at least $n$ vertices. Specifically, choose $H_{m+k,m}$, where $k=0$ if $r=0$, $k=1$ if $0 < r \leq m$, and $k=2$ if $m < r \leq 2m$. 
    
    Next, contract $r \bmod m$ path arcs in the top row of $H_{m+k,m}$ to obtain a directed complete graph with exactly $n$ vertices. To obtain an oriented graph, we must delete one arc in each pair of double arcs. For each pair of double arcs between two vertices in the same column, choose either arc to delete. Also, delete arcs $(b,a)$ such that $(a,b)$ is a path arc. We are now left with an orientation $\vec{K}_n$ of $K_n$ attained from $H_{m+k,m}$ by contracting path arcs and deleting non-path arcs, which implies that $\throt{(\vec{K}_n)} \leq 2m + k - 1$ by Theorem \ref{charthm}. For any possible $m$ and $k$, this simplifies to $\throt{(\vec{K}_n)} \leq \ceil{2\sqrt{n} - 1}.$ An example of this construction with $n = 5$ is illustrated in Figure \ref{completeMin}.
    
    Since $\ceil{2\sqrt{n} - 1}$ is the lower bound of throttling number of any oriented graph, the above inequality becomes an equality. We have found orientations of $K_n$ that achieve both the smallest and largest possible throttling number, so the OTI of a complete graph is maximum.
    \end{proof}
    
    \begin{figure}[H]
       \centering
       \begin{subfigure}[b]{0.4\textwidth}
            \hspace{1.2 cm}
           \includegraphics[width=4cm]{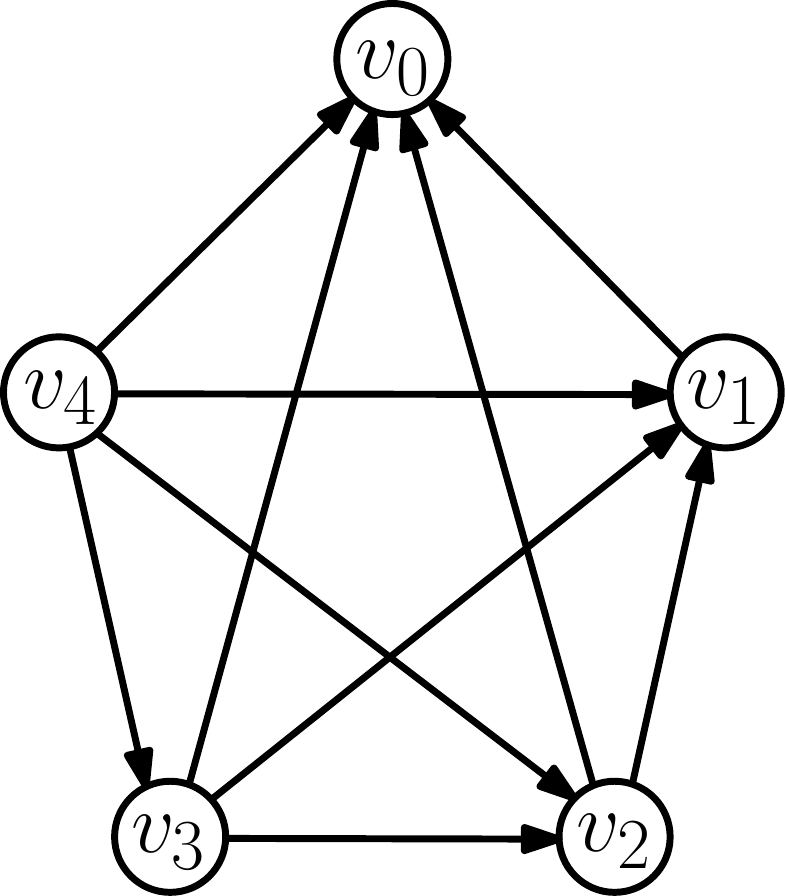}
           \caption{This oriented graph has throttling number 5.}\label{completeMax}
       \end{subfigure}
       \hspace{1 cm}
       \begin{subfigure}[b]{0.4\textwidth}
            \hspace{1.5 cm}
           \includegraphics[width=3.25cm]{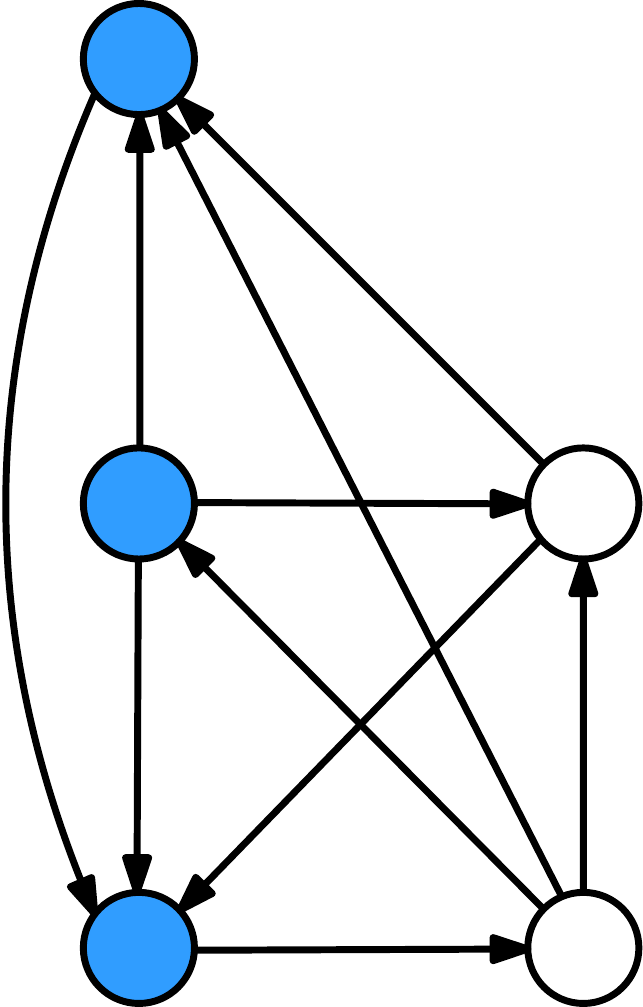}
           \caption{This oriented graph has throttling number $\ceil{2 \sqrt{5} - 1} = 4$.}\label{completeMin}
        \end{subfigure}
        \caption{Orientations of $K_5$ which achieve the maximum and minimum throttling numbers are shown.}\label{completeMaxMin}
    \end{figure}

While complete graphs have orientations that achieve both the maximum and minimum throttling number, this is not true for all types of graphs. In fact, paths have orientations that achieve the minimum throttling number, but not the maximum.

\begin{rem}\label{pathstuff}
By Proposition \ref{subset}, $m \leq \throt(P_n)$ where $\oti(P_n)=[m,M]$. Since $\throt(P_n)$ already achieves the lower bound of $\ceil{2\sqrt{n}-1}$ by \cite{BY}, it must be true that $m = \ceil{2\sqrt{n}-1}$ since $m < \ceil{2\sqrt{n}-1}$ is impossible. Thus, a path on $n$ vertices has an orientation $\vec{P}_n$ such that $\throt(\vec{P}_n) = \ceil{2\sqrt{n}-1}$. Additionally, by Theorem \ref{fullthrot}, the throttling number of any orientation of $P_n$ is bounded above by \[ \ceil{2\sqrt{n}-1}+\floor{\frac{E(P_n)}{2}} = \ceil{2\sqrt{n}-1}+\floor{\frac{n-1}{2}}. \]
\end{rem}

By Remark \ref{pathstuff}, it is impossible for an orientation of $P_n$ to achieve throttling number $n$ for $n \geq 14$. Furthermore, we can provide families of graphs that, unlike complete graphs and paths, do not have orientations that achieve the minimum possible throttling number. To do so, the following proposition is useful.

\begin{prop}\label{leaves}
    Let $\vec{G}$ be an oriented graph with at least one leaf. If $u \in V(\vec{G})$ is adjacent to $k$ leaves, a zero forcing set of $\vec{G}$ must contain at least $k - 1$ leaves adjacent to $u$.
\end{prop}

\begin{proof}
    Suppose for the sake of contradiction that a zero forcing set $B \subseteq V(\vec{G})$ contains $k-2$ or fewer leaves adjacent to $u$. This means at least two leaves $v,w$ adjacent to $u$ begin white. We can assume $v$ and $w$ are both sinks, as a zero forcing set necessarily contains all sources. Since $v$ and $w$ are leaves, they both can only be forced by $u$. However, $u$ can only force one vertex, so it cannot force both $v$ and $w$. Therefore, $B$ is not a zero forcing set, a contradiction. Thus, $B$ contains at least $k - 1$ leaves adjacent to $u$.
\end{proof}

While paths can be oriented to achieve the minimum throttling number but not the maximum, oriented stars only achieve the maximum throttling number.

\begin{cor}
    The throttling number of an oriented star on $n$ vertices is $n$.
\end{cor}

\begin{proof}
    By definition, a star on $n$ vertices contains $n-1$ leaves. By Proposition \ref{leaves}, we must initially color $n-2$ leaves blue. If the last remaining leaf is a source, it must be included in the zero forcing set as well, and since there is only one white vertex left, no simultaneous forces can be performed. If the last remaining leaf is a sink, we know that the only vertex that can force the leaf is the central vertex. Thus, both vertices cannot be forced in the same time step because the central vertex must be blue before forcing the leaf. Hence, there are no simultaneous forces in all cases, so the throttling number is $n$.
\end{proof}

This leads us to the question whether there are graphs that cannot be oriented to achieve the upper and the lower bounds on the throttling number. To answer this question positively, we must first establish a more general understanding of the impact of leaves on zero forcing.

\begin{cor}\label{superleaves}
    Let $\vec{G}$ be an oriented graph and $X$ be the set of vertices in $\vec{G}$ adjacent to at least one leaf. Let $x = |X|$ and $y$ be the number of leaves in $\vec{G}$. If $B$ is a zero forcing set of $\vec{G}$, then $|B| \geq y - x$.
\end{cor}

\begin{proof}
    For some vertex $v \in X$, let $k_{v}$ be the number of leaves to which $v$ is adjacent. From Proposition \ref{leaves}, we know that a zero forcing set of $\vec{G}$ must contain at least $k_{v}-1$ of the leaves adjacent to $v$. Summing over all vertices in $X$, we have
    \[\sum\limits_{v \in X} (k_{v} - 1) = \sum\limits_{v \in X} k_{v} - \sum\limits_{v \in X} 1 = y - x\]
    leaves that must be in a zero forcing set of $\vec{G}$. Note that no leaves are double-counted since a leaf is only adjacent to a single vertex in $X$, by definition.
\end{proof}

% It is important to note that Corollary \ref{superleaves} is only useful for graphs that have many leaves. On a path with four or more vertices, for instance, there are always exactly two vertices adjacent to leaves and exactly two leaves, so Corollary \ref{superleaves} gives 0 as a lower bound of the size of a minimum zero forcing set, which is not at all helpful. 

While Corollary \ref{superleaves} is not very helpful for graphs with relatively few leaves, such as paths, a double star on $n$ vertices is an example of a graph for which it is useful. In double stars, $y = n-2$ and $x = 2$, giving $n - 4$ as a lower bound on the size of a minimum zero forcing set. 

Corollary \ref{superleaves} is additionally useful for a family of graphs that we call \emph{augmented double stars}. To obtain this graph, let $G$ be any double star on $n-1$ vertices. Label the internal vertices $a$ and $b$. Delete the edge $ab$ and add a vertex $w$ and the edges $aw$ and $bw$. An example of an augmented double star is shown in Figure \ref{augstardef}.

\begin{figure}[H]
    \centering
    \includegraphics[scale = 0.4]{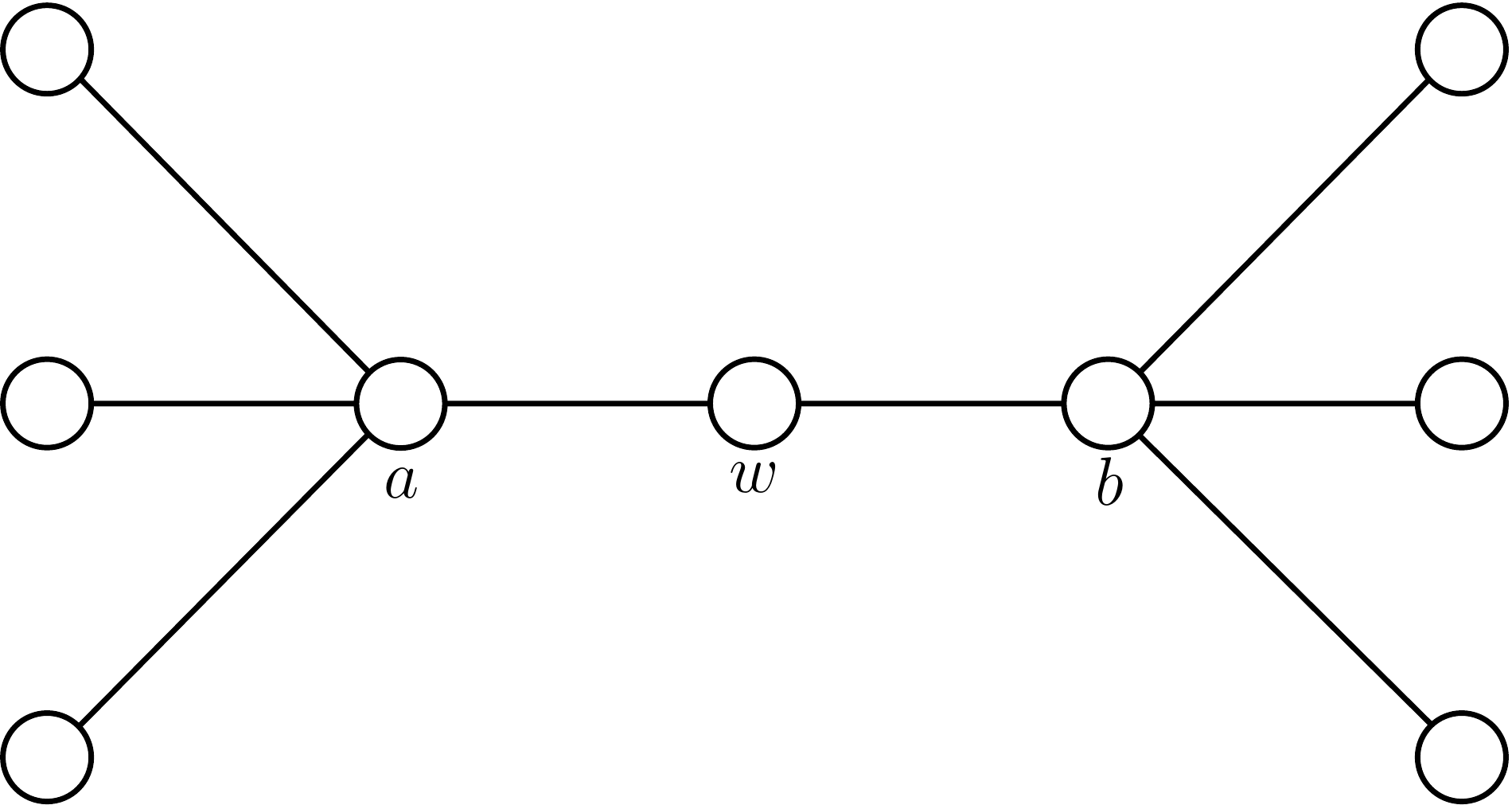}
    \caption{An augmented double star on 9 vertices is shown.} \label{augstardef}
\end{figure}

Interestingly, unlike complete graphs, sufficiently large augmented double stars have no orientations that achieve the upper or the lower bound for throttling number.

\begin{thm}
    Let $G$ be an augmented double star on $n$ vertices with $n \geq 12$, and let $\vec{G}$ be any orientation of $G$. Then $\ceil{2\sqrt{n}-1} < \throt(\vec{G}) < n$.
\end{thm}

\begin{proof}
    From Corollary \ref{superleaves}, since $\vec{G}$ has $n-3$ leaves and two vertices adjacent to these leaves, a lower bound on the number of vertices in a zero forcing set is $n-3 - 2 = n-5$. Since $n \geq 12$, we have $n - 5 > \ceil{2\sqrt{n} - 1}$, so $\ceil{2\sqrt{n}-1} < \throt(\vec{G})$.
        
    To upper bound $\throt(\vec{G})$, first recall that there are two vertices adjacent to leaves and one vertex adjacent to both of those vertices. Label them $a$, $b$, and $w$, respectively. Let $a'$ be a leaf adjacent to $a$ and $b'$ be a leaf adjacent to $b$. Color all vertices in $\vec{G}$ except $a,a',b$, and $b'$. If $a'$ is a source, color it blue; otherwise, color $a$ blue. Similarly, if $b'$ is a source, color it blue; otherwise, color $b$ blue. The remaining two white vertices can be forced in a single time step, meaning $\throt(\vec{G}) \leq (n-2)+1 < n$.
\end{proof}

In this section, we introduced the concept of an orientation throttling interval, proved its fullness for all simple graphs, and provided examples of graph families with different orientation throttling intervals. While paths are one of the most basic types of graph families, determining the maximum value of the OTI of a path actually proves to be rather difficult. In Section \ref{altpaths}, we investigate this further.

%%%%%%%%%%%%%%%%%%%%%%%%%%%%%
\section{Throttling on alternating paths} \label{altpaths}

In Remark \ref{pathstuff}, we found an upper bound on the throttling number of an oriented path. Note that a tight upper bound must be at least the throttling number of any particular orientation. This section is dedicated to finding an exact formula for the throttling number of a specific orientation of paths. %, which we speculate to be a tight upper bound for the throttling number of an oriented path.

%However, improving that bound has proven to be extremely difficult.

\begin{defn}
    An \emph{alternating path} on $n$ vertices is an orientation of $P_n$ where every vertex is either a source or a sink. Note that, when $n$ is even, the reversal of this graph is equivalent to the original graph, which is not the case when $n$ is odd.
\end{defn}

%We conjecture that for an undirected path, the alternating path realizes the maximum throttling number in the path's OTI.
We now compute an exact formula for the throttling number of an alternating path $\vec{P}_n$ in terms of $n$. To start this process, we utilize another known color change rule.

\begin{rem}\label{PSD}
    \emph{Positive semidefinite (PSD) throttling} is a variation of standard throttling that is studied in \cite{semidefinite}. While we are not concerned with the technical definition of the PSD color change rule, denoted $\Z_+$, we make use of how it works on an undirected path $P_n$. On $P_n$, the PSD color change rule can be simplified as follows: in each time step, all blue vertices force all of their adjacent white vertices simultaneously (note that each vertex is no longer limited to only performing one force). The PSD propagation time of a set of vertices $B$ in a graph $G$ is denoted $\ptp(G;B)$ and the PSD throttling number of $G$ is denoted $\thp(G)$.
\end{rem}

Remark \ref{PSD} allows us to quickly find the throttling number of an odd alternating path.

\begin{prop}\label{oddpathbounds}
    For an alternating path $\vec{P}_n$ where $n$ is odd, $\throt (\vec{P}_n) =\frac{n-1}{2} + \ceil{\sqrt{n+1} - \frac{1}{2}}.$
\end{prop}

\begin{proof}
    Without loss of generality, we consider the alternating path on $n$ vertices $\vec{P}_n$ where both endpoints are sinks. We can do this because $\throt(\vec{P}_n) = \throt(\vec{P}_n^T)$ by Theorem \ref{reversethrot}. The orientation $\vec{P}_n$ has a total of $\frac{n-1}{2}$ sources which must be colored blue initially. Note that the set of sources alone is not a zero forcing set since no forces can occur because each source is adjacent to two sinks. Thus, we must initially color some of the sinks blue as well. For this orientation, coloring a sink blue causes the adjacent sources to simultaneously force the sinks immediately to the left and right respectively (so long as they exist and have not yet been colored blue themselves). Ignoring the sources, this is an identical process to PSD forcing on an undirected path with $\frac{n+1}{2}$ vertices as described in Remark \ref{PSD}, where the vertices of this path represent the $\frac{n+1}{2}$ sinks in $\vec{P}_n$. Thus, the throttling number of $\vec{P}_n$ is the sum of the $\frac{n-1}{2}$ sources and the PSD throttling number of a path with $\frac{n+1}{2}$ vertices. By Theorem 3.2 from \cite{semidefinite}, this gives us that $\throt (\vec{P}_n) = \frac{n-1}{2} + \ceil{\sqrt{n + 1} - \frac{1}{2}}$.
\end{proof}

While throttling an alternating path on an odd number of vertices is relatively straightforward, the structure of an alternating path on an even number of vertices makes the problem more difficult. We construct an auxiliary path to help us overcome this issue.

\begin{rem}\label{aux}
     Note that an alternating path $\vec{P}_n$ on $n$ vertices where $n$ is even has $\frac{n}{2}$ sources, all of which must be blue initially in order to force the entire path. As in the proof of Proposition \ref{oddpathbounds}, throttling the remaining $\frac{n}{2}$ vertices is akin to PSD throttling on an undirected path with $\frac{n}{2}$ vertices. However, in this case, the endpoint that is a source will begin forcing the remaining $\frac{n}{2}$ white sinks immediately. To account for this, we construct an undirected path $P'$ on $1+\frac{n}{2}$ vertices and initially color one of the endpoints blue, which we call $u_0$. The vertex $u_0$ corresponds to the endpoint in $\vec{P_n}$ that is a source, while all of the other vertices in $P'$ represent the sinks in $\vec{P}_n$.  With this construction, any standard zero forcing process on $\vec{P}_n$ is equivalent to some PSD forcing process on $P'$ where $u_0$ is blue initially. An example of this equivalence is depicted in Figure \ref{auxfig} with corresponding zero forcing sets shown in blue. Thus, we can determine $\throt(\vec{P}_n)$ by first minimizing $|B'| + \ptp(P';B')$ where $B'$ is a PSD zero forcing set of $P'$ that contains $u_0$, then adding $\frac{n}{2}$ to account for the remaining sources and subtracting one to avoid counting $u_0$ twice.
\end{rem}

\begin{figure}[H]
    \centering
    \includegraphics{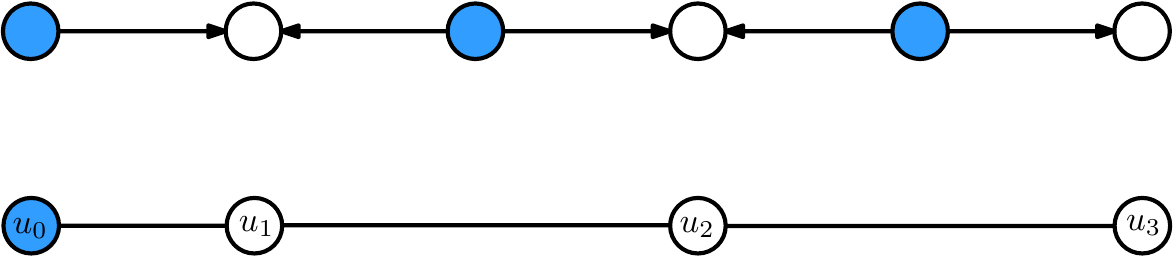}
    \caption{The top graph is the alternating path $\vec{P}_6$. The bottom graph is the corresponding path $P'$, as constructed in Remark \ref{aux}.}\label{auxfig}
\end{figure}

Next, we compute lower and upper bounds for the throttling number of an even alternating path, which we later show to be equivalent.

\begin{prop}\label{evenpathbounds2}
    Suppose $n$ is an even positive integer and $p=\frac{\sqrt{n+1}-1}{2}$. Then, an alternating path $\vec{P}_n$ satisfies $\throt(\vec{P}_n) \geq \frac{n}{2}+\ceil{2p}.$ 
\end{prop}

\begin{proof}
    By Remark \ref{aux}, it is sufficient to obtain a lower bound for $|B'| + \ptp(P'; B')$ where $P'$ is the auxiliary path for $\vec{P}_n$ and $B'$ is a PSD zero forcing set that contains $u_0$. Let $x = |B' \setminus \{u_0\}| \geq 0$ be the number of vertices that are initially colored blue in $P'$ other than $u_0$. Note that $x$ is also the number of sinks in $\vec{P}_n$ initially colored blue and the total number of initially blue vertices in $\vec{P}_n$ is $\frac{n}{2}+x$. Since $u_0$ creates one forcing chain in $P'$ and each of the $x$ other blue vertices can create at most two forcing chains, the largest number of vertices that can be forced using PSD zero forcing on $P'$ during a single time step is $2x+1$.
    % Let $B$ be the set of $\frac{n}{2}+x$ vertices initially colored blue. Note that $B$ is a zero forcing set of $\vec{P}_n$ because it contains the blue endpoint that corresponds to $u_0$ in $P'$, which can force the entire graph.   
    
    Let $t' = \ptp(P';B')$. As $B'$ forces the entire graph in $t'$ time steps, it follows that $x+1+t'(2x+1)\geq |V(P')| = \frac{n}{2}+1$. Solving for $x$ yields \[x \geq \frac{\frac{n}{2}-t'}{2t'+1}. \] 
    To find the throttling number on the throttling number of $P'$, we minimize $x + t'$. By the previous inequality, \[ x + t' ~\geq~ \frac{\frac{n}{2}-t'}{2t'+1}+t' ~:=~ f(t').\] This means that $x+t' \geq \underset{a \in \mathbb{R}}{\min}{\{f(a)\}}$. If we differentiate $f(t')$ with respect to $t'$, we find that $f$ has a critical point at $t' = \frac{\sqrt{n+1}-1}{2} = p$. Taking the second derivative of $f$ and substituting $p$ for $t'$ yields a positive value, meaning $t' = p$ is a minimum. Therefore, \[
        |B'| - 1 + t' ~=~ x+t' ~\geq~ \frac{\frac{n}{2}-p}{2p+1}+p ~=~ 2p
    \] which means that $2p +1$ is a lower bound for $|B'| + \ptp(P'; B')$. Since $|B'| + \ptp(P'; B')$ is necessarily an integer, we have that \begin{eqnarray}\label{araay}
        |B'| + \ptp(P'; B') \geq \ceil{2p + 1} = \ceil{2p} + 1.
    \end{eqnarray} Now suppose $B \subseteq V(\vec{P}_n)$ satisfies $\throt(\vec{P}_n) = |B| + \pt(\vec{P}_n; B)$ and let $B'$ be the corresponding PSD zero forcing set in $P'$. Therefore, by Remark \ref{aux} and the inequality in (\ref{araay}), \[ \throt{(\vec{P}_n)} = |B| + \pt(\vec{P}_n;B) ~=~ |B'| + \frac{n}{2} - 1 + \ptp(P'; B') \] \[\geq \ceil{2p} + 1 + \frac{n}{2} - 1 ~=~ \frac{n}{2} + \ceil{2p}. \hspace{1.3 cm} \qedhere \]
\end{proof}

\begin{prop}\label{evenpathbounds1}
     Suppose $n$ is an even positive integer and $p=\frac{\sqrt{n+1}-1}{2}$. Then, the alternating path $\vec{P}_n$ satisfies \[ \throt(\vec{P}_n) \leq \frac{n}{2} + \ceil{\frac{\frac{n}{2} - \ceil{p}}{2\ceil{p} + 1}} + \ceil{p}.\]
\end{prop}

\begin{proof}
    As described in Remark \ref{aux}, construct the auxiliary path $P'$. For each $1 \leq i \leq \frac{n}{2}$, let $u_i$ be the $i^{\text{th}}$ vertex after $u_0$ in $P'$ (see Figure \ref{auxfig}). We now construct a PSD zero forcing set $B'$ of $P'$ with $\ptp(P'; B') \leq \ceil{p}$ as follows. Starting with $u_0$, color every $\left( 2\ceil{p} + 1 \right)^{\text{th}}$ vertex of $P'$ blue, i.e., $u_0$, $u_{2\ceil{p} + 1}$, \ldots, $u_j$ where $j=m(2\ceil{p} + 1)$ and \[m = \floor{\frac{\frac{n}{2}}{2\ceil{p} + 1}} \geq 0. \] This leaves a tail of $k = \frac{n}{2} - j$ white vertices after $u_j$ where $0 \leq k \leq 2\ceil{p}$. If $\ceil{p} < k \leq 2\ceil{p}$, color the endpoint $u_{\frac{n}{2}}$ of $P'$ blue. Let $B'$ be the resulting set of blue vertices in $P'$. By construction of $B'$ (see Example \ref{ex:evenaltconstruction}), all forcing chains are of length at most $\ceil{p}$ where each blue endpoint begins forcing in one direction while the other blue vertices begin forcing in two directions. Thus, $\ptp(P'; B') \leq \ceil{p}$.
    
    Next, we consider the size of $B'$. Excluding $u_0$ and the $\ceil{p}$ vertices in the forcing chain started by $u_0$, there are $\frac{n}{2} - \ceil{p}$ remaining vertices in $P'$. By how $B'$ is constructed, the first $m-1$ vertices that were colored blue after $u_0$ are each the unique blue vertex in a set of $2\ceil{p}+1$ consecutive vertices. If $0 \leq k \leq \ceil{p}$ or $\ceil{p} < k \leq 2\ceil{p}$, then 
    \[m-1 < \frac{ \frac{n}{2} - \ceil{p}}{2 \ceil{p} + 1} \leq m ~~~~\text{or}~~~~ m < \frac{\frac{n}{2} - \ceil{p}}{2\ceil{p} + 1} < m + 1, \]
    respectively. Also, if $0 \leq k \leq \ceil{p}$, then we have $m+1$ initially blue vertices in $P'$; otherwise, we have $m+2$ initially blue vertices. In both cases, the number of blue vertices is 
    \[ \ceil{\frac{\frac{n}{2}-\ceil{p}}{2\ceil{p} +1}} + 1, \] which corresponds to the number of sinks we color in $\vec{P}_n$ and its endpoint that is a source. Let $B$ be the set of initially blue vertices in $\vec{P}_n$ that corresponds to $B'$ in $P'$. By Remark \ref{aux}, \[ \throt(\vec{P}_n) ~\leq~ |B| + \pt(\vec{P}_n; B) ~=~ |B'| + \frac{n}{2} - 1 + \pt(\vec{P}_n; B') \hspace{2 cm} \] \[ \hspace{2.1 cm} \leq~ \ceil{\frac{\frac{n}{2}-\ceil{p}}{2\ceil{p}+1}} + 1 + \frac{n}{2} - 1 + \ceil{p} ~=~ \frac{n}{2} + \ceil{ \frac{ \frac{n}{2} - \ceil{p}}{2\ceil{p}+1}} + \ceil{p}. \qedhere \]
\end{proof}

% \begin{eqnarray*}
%     \throt(\vec{P}_n) &\leq& |B| + \pt(\vec{P}_n; B) \\[.2 cm]
%     &=& |B'| + \frac{n}{2} - 1 + \pt(\vec{P}_n; B')\\[.2 cm]
%     &\leq& \frac{n}{2} - 1 + \ceil{\frac{\frac{n}{2}-\ceil{p}}{2\ceil{p}+1}} + 1 + \ceil{p} 
% \end{eqnarray*} \[ \hspace{.2 cm} =~ \frac{n}{2} + \ceil{ \frac{ \frac{n}{2} - \ceil{p}}{2\ceil{p}+1}} + \ceil{p}. \qedhere \]

\begin{ex}\label{ex:evenaltconstruction}
An example of the construction of $B'$ in the proof of Proposition \ref{evenpathbounds1} is shown in Figure \ref{evenaltpicture}. For $n = 16$, $\ceil{p} = 2$. Note that we start by coloring $u_0$ and skip over $2\ceil{p}$ vertices each time. Also, $u_j=u_5=u_{2\ceil{p}+1}$, so $m=1$ and $k=3>\ceil{p}$. Thus, we color $u_8$ at the end.
\end{ex}

\begin{figure}[H]
    \centering
    \includegraphics[scale=0.88]{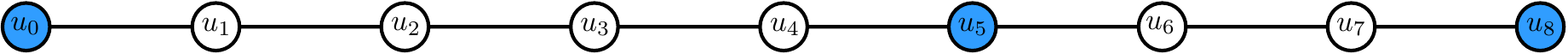}
    \caption{The auxiliary path $P'$ for $\vec{P}_{16}$ is shown with PSD zero forcing set $B$ such that $\ptp(\vec{P}_{16}; B) \leq \ceil{p} = 2$.}\label{evenaltpicture}
\end{figure}

We now have upper and lower bounds on the throttling number of an even alternating path, which we can show are equal to obtain an exact formula. To do so, we need the following fact.

\begin{rem}\label{knuth}
\cite[page 72]{knuth} Let $m$ and $n$ be integers such that $n > 0$. Then, for any $x \in \mathbb{R}$, \[\ceil{\frac{m + \ceil{x}}{n}} = \ceil{\frac{m + x}{n}}.\]
\end{rem}

\begin{thm}\label{evenpath}
    For an alternating path $\vec{P}_n$ where $n$ is even, $\throt (\vec{P}_n) = \frac{n}{2} + \ceil{\sqrt{n+1} - 1}$.
\end{thm}

\begin{proof}
    %Let $P=\left\{\frac{\sqrt{n+1}-1}{2} ~\Big|~ n ~\operatorname{mod}~ 2 = 0 \text{ and } n > 0\right\}$. 
    Let $n$ be a positive even integer and $p=\frac{\sqrt{n+1}-1}{2}$. Note that $n = 4p^2 + 4p$. From Propositions \ref{evenpathbounds2} and \ref{evenpathbounds1}, we have an upper and lower bound for $\throt(\vec{P}_n)$. Thus, it is sufficient to prove that these bounds are equal, which after algebraic manipulation and substitution is equivalent to proving
    
    % \[
    %     \ceil{ \frac{ \frac{n}{2} - \ceil{p}}{2 \ceil{p} + 1}} + \ceil{p} = \ceil{ 2p }.
    % \]
    % Also note that, by the definition of $p$ and simple algebraic manipulation, we have that $n = 4p^2 + 4p$, which makes the above statement equivalent to proving that 
    
    \begin{eqnarray} 
        \ceil{ \frac{ 2p^2+2p - \ceil{p}}{2 \ceil{p} + 1} + \ceil{p}} = \ceil{ 2p }. \label{equiv}
    \end{eqnarray}
    
    We will first bound the value under the ceiling on the left side of equation (\ref{equiv}) below by $2p$. The lower bound of $2p$ can be shown with this series of equivalent statements:
    \begin{eqnarray}
        &&\frac{ 2p^2+2p - \ceil{p}}{2 \ceil{p} + 1} + \ceil{p} \geq 2p \label{equivlower}\\[.2 cm]
        &\iff&\frac{ 2p^2+2 \ceil{p}^2-4p\ceil{p}}{2 \ceil{p} + 1} \geq 0 \nonumber\\[.2 cm]
        &\iff&2p^2+2 \ceil{p}^2-4p\ceil{p} \geq 0 \hspace{1 cm} \text{(since $2\ceil{p}+1> 0$)}\nonumber\\[.2 cm]
        &\iff&2(p-\ceil{p})^2 \geq 0.\nonumber
    \end{eqnarray}
    We know $2(p-\ceil{p})^2 \geq 0$ is always true, so the lower bound of $2p$ given by (\ref{equivlower}) must be true.
    
   Now, we split into 2 cases. First, suppose $\ceil{2p}=\ceil{2\ceil{p}}$. This case motivates us to bound the expression on the left side of (\ref{equivlower}) above by $2\ceil{p}$. This upper bound is clear: \[\frac{ 2p^2+2p - \ceil{p}}{2 \ceil{p} + 1} + \ceil{p} \leq \frac{2\ceil{p}^2+2\ceil{p}-\ceil{p}}{2\ceil{p}+1}+\ceil{p}=2\ceil{p},\]
   since $\ceil{p} \geq p$ and $2\ceil{p} + 1> 0$. Thus, taking this with (\ref{equivlower}) gives us   
   \[2p\leq\frac{ 2p^2+2p - \ceil{p}}{2 \ceil{p} + 1} + \ceil{p} \leq 2\ceil{p}.\] This equivalence follows since $\ceil{2p}=\ceil{2\ceil{p}}$, which proves (\ref{equiv}): \[\ceil{2p}=\ceil{\frac{ 2p^2+2p - \ceil{p}}{2 \ceil{p} + 1} + \ceil{p}}=\ceil{2\ceil{p}}.\]
    
    Alternatively, suppose $\ceil{2p}=\ceil{2\ceil{p}-1}$, which is equivalent to  $-1<p-\ceil{p}\leq -\frac{1}{2}$. This case motivates us to bound the expression from the left side of (\ref{equivlower}) above by $2\ceil{p}-1$. We do this with the following equivalent statements:
    \begin{eqnarray}
        &&\frac{ 2p^2+2p-p-1}{2 \ceil{p} + 1} + \ceil{p} \leq 2\ceil{p} -1 \label{equivupper2}\\[.2 cm]
        &\iff&\frac{ 2p^2+p-2 \ceil{p}^2 + \ceil{p}}{2 \ceil{p} + 1} \leq 0  \nonumber \\[.2 cm]
        &\iff& 2p^2+p-2 \ceil{p}^2 + \ceil{p} \leq 0  \nonumber \hspace{2 cm}\text{(since $2\ceil{p}+1>0$)} \\[.2 cm]
        &\iff& (p + \ceil{p})(2(p-\ceil{p})+1) \leq 0  \nonumber \\[.2 cm]
        &\iff& 2(p-\ceil{p})+1 \leq 0.  \hspace{3.3 cm}\text{(since $p + \ceil{p}>0$)} \nonumber
    \end{eqnarray}
    We see that $2(p-\ceil{p})+1 \leq 0$ is always true since $-1<p-\ceil{p}\leq -\frac{1}{2}$, so (\ref{equivupper2}) is true. Additionally, observe that 
    \[ \ceil{\frac{ 2p^2+2p-\ceil{p}}{2 \ceil{p} + 1}}=\ceil{\frac{ 2p^2+2p+\ceil{-p-1}}{2 \ceil{p} + 1}}=\ceil{\frac{ 2p^2+2p-p-1}{2 \ceil{p} + 1}}.\]
    The first equality is due to the fact that $-\ceil{p}=\ceil{-p-1}$ when $-1<p-\ceil{p}\leq -\frac{1}{2}$. The second equality follows from applying Remark \ref{knuth}, which can be done since $2p^2+2p=\frac{n}{2} \in \ZZ$ and $2\ceil{p}+1 \in \ZZ^+$. This, along with the bounds in (\ref{equivlower}) and (\ref{equivupper2}), gives us
    
    \[\ceil{2p} \leq 
    \ceil{\frac{ 2p^2+2p - \ceil{p}}{2 \ceil{p} + 1} + \ceil{p}} = \ceil{\frac{ 2p^2+2p-p-1}{2 \ceil{p} + 1}+ \ceil{p}} \leq \ceil{2\ceil{p} -1}.\]
    Since $\ceil{2p}=\ceil{2\ceil{p} -1}$ in this case, we obtain the equality in (\ref{equiv}): \[\ceil{\frac{ 2p^2+2p - \ceil{p}}{2 \ceil{p} + 1} + \ceil{p}}=\ceil{2p}.\]
    
    \noindent Since (\ref{equiv}) is true in both cases, $\throt (\vec{P}_n) = \frac{n}{2} + \ceil{\sqrt{n+1} - 1}$.
    %, which is equivalent to showing the equality of our lower and upper bounds for $\throt (\vec{P}_n)$. 
    % Hence, we have \[ \throt (\vec{P}_n) = \frac{n}{2} + \ceil{\sqrt{n+1} - 1} \] when $n$ is even and $\vec{P}_n$ is an alternating path.
\end{proof}

Note that if $\vec{P}_n$ is an alternating path for some positive integer $n$, $\throt(\vec{P}_n)$ is a lower bound for the maximum value of the OTI of $P_n$. We conjecture that the alternating path achieves this maximum value and the following results are tools that we build which may aid in proving this conjecture.

\begin{prop}\label{sourcesinkflip}
    Let $\vec{G}$ be an oriented graph that has an arc $(u, v)$ such that $u$ is a source and $v$ is a sink. If the arc $(u,v)$ is flipped to obtain $\vec{G}_0$, then $\throt(\vec{G_0}) \leq \throt(\vec{G})$.
\end{prop}

\begin{proof}
     Let $B$ be a zero forcing set of $\vec{G}$, and let $\F$ be a set of forces of $B$ that achieves $\throt(\vec{G})=|B|+\pt(\vec{G};\F)$. Since $u$ is a source, $u \in B$. There are now two cases to consider.
    
    First, suppose $(u \to v) \notin \F$. Initially color $B$ blue in $\vec{G}_0$. Note that the only vertices with different out-neighborhoods in $\vec{G}_0$ are $u$ and $v$. Specifically, $u$ has lost an out-neighbor and $v$ has gained one, namely $u$. However, since $u \in B$, the set of white out-neighbors of $v$ remains the same. Also, each arc in $\F$ exists in $\vec{G}_0$. As a result, $\pt(\vec{G}_0; \F) \leq \pt(\vec{G}; \F)$, so $\throt{(\vec{G}_0)} \leq |B| + \pt(\vec{G}_0;\F) \leq |B| + \pt(\vec{G};\F) = \throt{(\vec{G})}$.
    
    Now, suppose $(u \to v) \in \F$, which implies $\pt(\vec{G}; B) \geq 1$. Initially color $\left( B \setminus \{u\} \right) \cup \{v\}$ blue in $\vec{G}_0$. As in the previous case, $u$ has lost a white out-neighbor and $v$ has gained a white out-neighbor, namely $u$. However, since $v$ was a sink in $\vec{G}$, the vertex $u$ is the only out-neighbor of $v$ in $\vec{G}_0$, so $v$ forces $u$ in the first time step. For any other vertex in $V(\vec{G}_0)$, its white out-neighborhood in $\vec{G}_0$ is a subset of its white out-neighborhood in $\vec{G}$ because none of these vertices have $u$ as an out-neighbor and $v$ starts blue. Thus, all forces that occurred on the first time step in $\vec{G}$ can still occur on the first time step in $\vec{G}_0$, excluding $u \to v$ but including $v \to u$. After the first time step, all remaining forces in $\F$ can occur on $\vec{G}_0$ without increasing propagation time. Thus, it follows that
    \begin{eqnarray*}
        \throt{(\vec
        G_0)} &\leq& | ( B \setminus \{u\} ) \cup \{v\} | + \pt \left(\vec{G}_0; ( \F \setminus \{u \to v\} ) \cup \{v \to u\} \right) \\
        &\leq& |B| + \pt(\vec{G}; \F) \\[.2 cm]
        &=& \throt{(\vec{G})}.
    \end{eqnarray*} All cases have been exhausted. Thus, $\throt{(\vec{G}_0)} \leq \throt{(\vec{G})}$.
\end{proof}

\begin{cor}\label{sourcesinkflipcor}
    Let $\vec{P}$ be an oriented path with $(u,v) \in E(\vec{P})$ such that both $u$ and $v$ are neither sources nor sinks. If the arc $(u,v)$ is flipped to obtain $\vec{P}'$, then $\throt(\vec{P}') \geq \throt(\vec{P})$.
\end{cor}

\begin{proof}
    Since $u$ and $v$ are neither sources nor sinks and $\vec{P}$ is a path, both $u$ and $v$ each have in-degree 1 and out-degree 1. 
    After flipping $(u, v)$, $u$ is now a sink and $v$ is now a source in $\vec{P}'$. By Proposition \ref{sourcesinkflip}, $\throt(\vec{P}) \leq \throt(\vec{P'})$ since $\vec{P}$ can be obtained from $\vec{P}'$ by flipping the arc $(v,u)$ between a source and a sink.
\end{proof}

%%%%%%%%%%%%%%%%%%%%%%%%%%%%%

\begin{section}{Concluding Remarks}\label{sec:conclusion}

     It is clear that the throttling number of any undirected graph is bounded below by the minimum throttling number of all of its orientations. Proposition \ref{subset} and Corollary \ref{subsetcor} allow us to bound the throttling number of an undirected graph $G$ above by the maximum throttling number of all of its orientations if $\throt(G) \leq \alpha(G) + 1$. However, it remains to be shown whether $\throt(G)$ is contained in the orientation throttling interval of $G$ whenever $\throt(G) > \alpha(G) + 1$.

    In Section \ref{altpaths}, we studied the alternating path $\vec{P}_n$, which we conjecture to achieve the maximum throttling number in $\oti(P_n)$. We have verified this computationally for $n \leq 14$ (see \cite{code}), but it still remains an open question whether this is true for paths of any length. Proposition \ref{sourcesinkflip} and Corollary \ref{sourcesinkflipcor} may be useful starting points since they characterize the behavior of the throttling number after performing certain types of arc flips. However, not all oriented paths can be obtained from alternating paths merely by performing these specific types of flips. Additionally, in many cases there exist multiple orientations of a path that achieve the maximum throttling number in the OTI. To aid in future computations, we share a public GitHub repository \cite{code} containing multiple \emph{Sage} programs which can calculate throttling number, propagation time, terminus, OTI, and other parameters for a given graph or digraph.
    
    Another question we have is whether we can generalize the alternating path conjecture to all bipartite graphs. In other words, is it true that for any bipartite graph, the upper bound of that graph's OTI is achieved when every vertex is either a source or a sink? Note that, in a bipartite graph, it is possible obtain such an orientation by directing all arcs from one part to the other.
    
    Throttling has also been studied as a \emph{forbidden subgraph problem} for undirected graphs in \cite{JK19}. Considering this problem for directed graphs, we found that if a graph $G$ on $n$ vertices has $C_5$, $K_2 \cup K_2$, $K_3$ $\square$ $P_2$, or a subgraph of $K_3$ $\square$ $P_2$ obtained by deleting $K_3$ edges, no orientation of $G$ has throttling number $n$. However, a complete characterization of forbidden subgraphs does not yet exist. 

\end{section}
%%%%%%%%%%%%%%%%%%%%%%%%%%%%%%
\begin{section}{Acknowledgements}
This work was carried out as part of the 2020 SMALL REU, supported by NSF-020262 and by Williams College.
\end{section}
%%%%%%%%%%%%%%%%%%%%%%%%%%%%%%%%%%%%%%

%%%%%%%%%%%%%%%%%%%%%%%%%%%%%%%%%%%%%%%
\end{document}